\documentclass[12pt,a4paper,twoside,reqno,amssymb]{amsart}

\usepackage[left=2.5cm,top=3cm,right=2.5cm,bottom=3cm]{geometry}

\title[Higgs bundles for the Lorentz group]{Higgs bundles for the Lorentz group}

\date{31 March 2010}

\author[M. Aparicio Arroyo \and O. Garc\'{\i}a-Prada]{Marta Aparicio Arroyo
\and Oscar Garc\'{\i}a-Prada}

\address{Instituto de Ciencias Matem\'aticas CSIC-UAM-UC3M-UCM\\ Serrano 121\\ 28006 Madrid, Spain}

\email{oscar.garcia-prada@uam.es}

\address{Instituto de Ciencias Matem\'aticas CSIC-UAM-UC3M-UCM\\ Serrano 123\\ 28006 Madrid, Spain}

\email{marta.aparicio@mat.csic.es}

\thanks{Partially supported by the Spanish Ministerio de Ciencia e
Innovaci\'on (MICINN) under grant~MTM2007-67623. }

\usepackage{amsmath}
\usepackage{amssymb}
\usepackage{amsfonts}
\usepackage{amsthm}
\usepackage{amscd}
\usepackage{mathrsfs}
\usepackage{graphicx}
\usepackage{xypic}
\usepackage{mathdots}
\usepackage{hyperref}

\usepackage{hyperref,url}

\usepackage{amssymb,latexsym}

\usepackage{amsmath,amsthm}

\usepackage[latin1]{inputenc}

\usepackage{enumerate}

\usepackage[all]{xy} %\CompileMatrices

\theoremstyle{plain}

\newtheorem{theorem}{Theorem}[section]

\newtheorem{lemma}[theorem]{Lemma}

\newtheorem{corollary}[theorem]{Corollary}

\newtheorem{proposition}[theorem]{Proposition}

\theoremstyle{definition}

\newtheorem{definition}[theorem]{Definition}

\newtheorem{definition-theorem}[theorem]{Definition-Theorem}

\theoremstyle{remark}

\newtheorem{remark}[theorem]{Remark}

\numberwithin{equation}{section} 

\newtheorem*{proposition*}{\textbf{\emph{Proposition}}}
\newtheorem*{teo*}{\textbf{\emph{Theorem}}}

\DeclareMathOperator{\End}{End} 
\DeclareMathOperator{\ad}{ad}\DeclareMathOperator{\Hom}{Hom}
\DeclareMathOperator{\Ad}{Ad}\DeclareMathOperator{\Tr}{Tr}
\DeclareMathOperator{\Aut}{Aut}

\DeclareMathOperator{\GL}{GL}\DeclareMathOperator{\SL}{SL}
\DeclareMathOperator{\Sp}{Sp}\DeclareMathOperator{\SO}{SO}

\DeclareMathOperator{\Spin}{Spin}

\DeclareMathOperator{\U}{U}

\DeclareMathOperator{\rk}{rk}
\DeclareMathOperator{\vol}{vol}

\begin{document}

\begin{abstract}
Using the Morse-theoretic methods introduced by Hitchin, we prove that
the moduli space of $\SO_0(1,n)$-Higgs bundles when $n$ is odd has two
connected components.
\end{abstract}

\maketitle

%\tableofcontents

\section*{Introduction}

Let $G$ be a real semisimple Lie group and let $H\subseteq G$ be a
maximal compact subgroup. Let
$\iota:H^\mathbb{C}\rightarrow\GL(\mathfrak{m}^\mathbb{C})$ be the
complexified isotropy representation defined in terms of the
Cartan decomposition of the Lie algebra of $G$.
Let $X$ be a compact Riemann surface of genus $g\geq 1$.
A \emph{$G$-Higgs
bundle} over $X$ is a pair $(E,\varphi)$
consisting of a principal $H^\mathbb{C}$-bundle $E$ over $X$ and a
holomorphic section $\varphi$ of the bundle associated to $\iota$
twisted by the canonical line bundle of $X$. For these objects there is
a notion of (poly)stability that allows to  construct the  moduli
space of isomorphism classes of polystable $G$-Higgs bundles.
Higgs bundles were introduced by Hitchin in \cite{Hi1,Hi0} when
$G$ is complex and in \cite{Hi2} when $G$ is the split real form
of a complex semisimple Lie group. Other real forms, especially of
Hermitian type have been studied in \cite{BG-PG1,BG-PG3,G-PGM} and other
papers.

In \cite{tesis} a systematic study has been initiated for
$G=\SO_0(p,q)$ --- the connected component of the identity of
$\SO(p,q)$.  In this paper we report on the solution  to the
problem of counting the number of connected components of the
moduli space of polystable $\SO_0(1,n)$-Higgs bundles when $n$
is odd. We prove the following.

\begin{teo*}[see Theorem \ref{SO(1,2m+1)components}]
\emph{The moduli space of $\SO_0(1,n)$-Higgs bundles when  $n>1$
is odd has two  connected components.}
\end{teo*}

An important motivation to study $G$-Higgs bundles comes from
their relation with  re\-pre\-sentations of the fundamental group
of the surface $X$ in $G$. Namely, for a semisimple algebraic Lie
group $G$ we say that a representation of $\pi_1(X)$ in $G$ ---
that is a homomorphism of $\pi_1(X)$ in $G$ --- is reductive if
the Zariski closure of its image is a reductive group. The moduli
space of equivalence classes of reductive representations is an
algebraic variety \cite{G1}. Non-abelian Hodge theory
\cite{BG-PM,C1,D,G-PGM,Hi1,S1,S2} says precisely  that this
variety is homeomorphic to the moduli space of polystable
$G$-Higgs bundles. We thus have the following as a corollary of
our main theorem.

\begin{teo*}
\emph{The moduli space of reductive representations of the
fundamental group of an orientable compact surface in
$\SO_0(1,n)$ when $n>1$ is odd  has two  connected components.}
\end{teo*}

The main tool to prove our result is the use of the
Morse-theoretic techniques introduced by Hitchin \cite{Hi1,Hi2}.
These techniques have by now been  used to count the number of
connected components of the moduli space of $G$-Higgs bundles for
several groups (see e.g. \cite{BG-PG1,BG-PG2,Gth2,G-PM,G-PGM,O}).
A main step
is to identify the critical subvarieties of the Hitchin-Morse
function defined by the $L^2$-norm of the Higgs field. This has been carried out in \cite{tesis} in full
generality for $\SO_0(p,q)$. Now, the problem of identifying the
local minima --- which is what allows the counting of connected
components --- in general is far more involved technically  than
for  the other groups studied in the literature. This is however
possible for $\SO_0(1,n)$ when $n$ is odd.  The main technical bulk
of the paper is devoted
to identifying in this case, first the smooth minima in the moduli space,
 and then
the possibly singular points, which  consist of stable but not simple Higgs
bundles and strictly polystable Higgs bundles.
We expect
that our results may be of  interest both in geometry
and physics since  $\SO(1,n)$  is the Lorentz group of special
relativity and its adjoint form is the group of isometries of real
hyperbolic space.

\noindent
{\bf Acknowledgements.} We thank Steven Bradlow, Peter Gothen, Nigel
Hitchin and Ignasi Mundet i Riera for useful discussions.
The first author thanks
the Max Planck Institute for Mathematics in Bonn --- that he was
visiting when this paper
was completed --- for support.

\section{$\SO_0(1,n)$-Higgs bundles}

Let $X$ be a compact Riemann surface. Let $G$ be a real semisimple
Lie group, $H$ be a maximal compact subgroup of $G$ and
$H^\mathbb{C}$ be its complexification. Let
$$\iota:H^\mathbb{C}\rightarrow\GL(\mathfrak{m}^\mathbb{C}),$$ be
the complexified isotropy representation, defined in terms of the
Cartan decomposition $\mathfrak{g}=\mathfrak{h}+\mathfrak{m}$ of
the Lie algebra of $G$ and using the fact that
$[\mathfrak{h},\mathfrak{m}]\subseteq\mathfrak{m}$.

\begin{definition}\label{higgsbundledefinition}
A \textbf{$G$-Higgs bundle} is a pair $(E,\varphi)$ where $E$ is a
principal $H^\mathbb{C}$-bundle over $X$ and $\varphi$ is a
holomorphic section of the vector bundle
$E(\mathfrak{m}^\mathbb{C})\otimes
K=(E\times_\iota\mathfrak{m}^\mathbb{C})\otimes K$, where $K$ is
the canonical line bundle over $X$. The section $\varphi$ is
called the \textbf{Higgs field}.

%Two $G$-Higgs bundles $(E,\varphi)$ and $(E',\varphi')$ are
%isomorphic if there is an isomorphism $E\cong E'$ which takes
%$\varphi$ to $\varphi'$ under the induced isomorphism
%$E(\mathfrak{m}^\mathbb{C})\cong E'(\mathfrak{m}^\mathbb{C})$.
\end{definition}

When $G$ is a real compact reductive Lie group, the Cartan
decomposition of the Lie algebra is $\mathfrak{g}=\mathfrak{h}$
and then the Higgs field is equal to zero. Hence, a $G$-Higgs
bundle is in fact a principal $G^\mathbb{C}$-bundle.

If $G$ is a complex Lie group, we consider the underlying real Lie
group $G^\mathbb{R}$. In this case, the complexification
$H^\mathbb{C}$ of a maximal compact subgroup is again the Lie
group $G$ and since
$$\mathfrak{g}^\mathbb{R}=\mathfrak{h}+i\mathfrak{h},$$ the isotropy
representation coincides with the adjoint representation of $G$ on
its Lie algebra.

The special orthogonal group $\SO(1,n)$ is the subgroup of
$\SL(n+1,\mathbb{R})$ consisting of all linear transformations of
a $n+1$ dimensional real vector space which leave invariant a
non-degenerate symmetric bilinear form of signature $(1,n)$. Using
the standard non-degenerate symmetric bilinear form of signature
$(1,n)$ on $\mathbb{R}^{n+1}$
$$\epsilon(x,y)=-x_1 y_1+x_{2}y_{2}+\cdots
+x_{n+1} y_{n+1},$$ this means that,
$$\SO(1,n)=\{A\in\SL(n+1,\mathbb{R})\mid
A^tI_{1,n}A=I_{1,n}\},$$ where $I_{1,n}=\left(%
\begin{array}{cc}
  -1 &  \\
   & I_n \\
\end{array}%
\right)$.

The Lie group $\SO(1,n)$ is a non-compact real form of
$\SO(n+1,\mathbb{C})$. It has dimension $n(n+1)/2$, is semisimple
for $n\geq 2$ and has two connected components. Let $\SO_0(1,n)$
be the connected component of the identity.

The Lie algebra of $\SO(1,n)$ and then of its identity component
$\SO_0(1,n)$ is $\mathfrak{so}(1,n)$, which has Cartan
decomposition
$$\mathfrak{so}(1,n)=\mathfrak{h}+\mathfrak{m},$$ where
$\mathfrak{h}=\mathfrak{so}(n)$ is the Lie algebra of the maximal
compact subgroup $\SO(1)\times\SO(n)$ of $\SO_0(1,n)$. If we use
the standard non-degenerate symmetric bilinear form of signature
$(1,n)$, we have that
\begin{align*}
\mathfrak{so}(1,n)& =
\{X\in\mathfrak{sl}(n+1,\mathbb{R})\mid
X^tI_{1,n}+I_{1,n}X=0\}\\
& =
\left\{\left(%
\begin{array}{cc}
  0 & X_2 \\
  X_2^t & X_3 \\
\end{array}%
\right)\mid X_3\text{ real skew-sym. of rank $n$,
$X_2\in\mathbb{R}^{n}$}\right\},
\end{align*}
 and then
$$\mathfrak{h}=\left\{\left(%
\begin{array}{cc}
  0 & 0 \\
  0 & X_3 \\
\end{array}%
\right)\mid X_3\in\mathfrak{so}(n)\right\},$$ and
$$\mathfrak{m}=\left\{\left(%
\begin{array}{cc}
  0 & X_2 \\
  X_2^t & 0 \\
\end{array}%
\right)\mid X_2\in\mathbb{R}^{n}\right\}.$$ The involution of
$\mathfrak{so}(n+1,\mathbb{C})$ that defines $\mathfrak{so}(1,n)$
as a real form is $\sigma(X)=I_{1,n}\bar{X}I_{1,n}$, that is
\begin{align*}
\mathfrak{so}(1,n)&=\{X\in\mathfrak{so}(n+1,\mathbb{C})\mid
I_{1,n}\bar{X}I_{1,n}=X\}\\
&=\{X\in\mathfrak{sl}(n+1,\mathbb{C})\mid
X+X^t=0,I_{1,n}\bar{X}I_{1,n}=X\}\\
&=\left\{\left(%
\begin{array}{cc}
  0 & iX_2 \\
  -iX_2^t & X_3 \\
\end{array}%
\right)\mid X_3\text{ real skew-sym. of rank $n$,
$X_2\in\mathbb{R}^{n}$}\right\}.
\end{align*} Observe that there is an isomorphism $$\left(%
\begin{array}{cc}
  0 & iX_2 \\
  -iX_2^t & X_3 \\
\end{array}%
\right)\rightarrow\left(%
\begin{array}{cc}
  0 & X_2 \\
  X_2^t & X_3 \\
\end{array}%
\right)=\left(%
\begin{array}{cc}
  -i & 0 \\
  0 & I_{n} \\
\end{array}%
\right)\left(%
\begin{array}{cc}
  0 & iX_2 \\
  -iX_2^t & X_3 \\
\end{array}%
\right)\left(%
\begin{array}{cc}
  i & 0 \\
  0 & I_{n} \\
\end{array}%
\right).$$ The Cartan decomposition of the complex Lie algebra is
$$\mathfrak{so}(n+1,\mathbb{C})=\mathfrak{so}(n,\mathbb{C})
\oplus\mathfrak{m}^\mathbb{C},$$ where
$$\mathfrak{m}^\mathbb{C}=\{\left(%
\begin{array}{cc}
  0 & X_2 \\
  -X_2^t & 0 \\
\end{array}%
\right)\mid X_2\in\mathbb{C}^{n}\},$$ and the complexified
isotropy representation is
$$\iota:\{1\}\times\SO(n,\mathbb{C})\rightarrow\GL(\mathfrak{m}^\mathbb{C}),$$
where \begin{eqnarray*}
  \iota\left(%
\begin{array}{cc}
  1 & 0 \\
  0 & b \\
\end{array}%
\right)\left(%
\begin{array}{cc}
  0 & X_2 \\
  -X_2^t & 0 \\
\end{array}%
\right)&=&\left(%
\begin{array}{cc}
  1 & 0 \\
  0 & b \\
\end{array}%
\right)\left(%
\begin{array}{cc}
  0 & X_2 \\
  -X_2^t & 0 \\
\end{array}%
\right)\left(%
\begin{array}{cc}
  1 & 0 \\
  0 & b^{-1} \\
\end{array}%
\right) \\
  &=&\left(%
\begin{array}{cc}
  0 & X_2b^{-1} \\
  -bX_2^t & 0 \\
\end{array}%
\right)\in\mathfrak{m}^\mathbb{C}.
\end{eqnarray*}

From Definition \ref{higgsbundledefinition}, an
\textbf{$\SO_0(1,n)$-Higgs bundle} is a pair $(E,\varphi)$
consisting of a holomorphic principal
$\SO(1,\mathbb{C})\times\SO(n,\mathbb{C})$-bundle $E$ over $X$ and
a holomorphic section $\varphi\in
H^0(E(\mathfrak{m}^\mathbb{C})\otimes K)$.

If $(E,\varphi)$ is an $\SO_0(1,n)$-Higgs bundle, the principal
$\SO(1,\mathbb{C})\times\SO(n,\mathbb{C})$-bundle $E$ is the
fibred product $$E=E_{\SO(1,\mathbb{C})}\times
E_{\SO(n,\mathbb{C})}$$ of two principal bundles with structure
groups $\SO(1,\mathbb{C})$ and $\SO(n,\mathbb{C})$ respectively.
Using the standard representations of $\SO(1,\mathbb{C})$ and
$\SO(n,\mathbb{C})$ in $\mathbb{C}$ and $\mathbb{C}^n$ we can
associate to $E$ a triple $(V,W,Q_W)$ where $V\cong\mathcal{O}$,
$W$ is a holomorphic vector bundle of rank $n$ and trivial
determinant, $Q_W:W\otimes W\rightarrow\mathbb{C}$ is a
non-degenerate symmetric quadratic form, which induces an
isomorphism $\xymatrix{q_W:W\ar[r]^-\sim &W^*}$.

The vector bundle $E(\mathfrak{m}^\mathbb{C})$ can be expressed in
terms of $V\cong\mathcal{O}$ and $W$ as follows:
$$E(\mathfrak{m}^\mathbb{C})=\{(\eta,\nu)\in\Hom(W,\mathcal{O})\oplus\Hom(\mathcal{O},W)\mid
\nu=-\eta^\top\},$$ where $\eta^\top=q_W^{-1}\circ\eta^t$,
$$\xymatrix{\mathcal{O}\ar[r]^{\eta^\top}\ar[dr]_{\eta^t} &W\ar[d]^{q_W}\\
& W^*,}$$ that is,
$E(\mathfrak{m}^\mathbb{C})\cong\Hom(W,\mathcal{O})$. Then, in
terms of vector bundles, the Higgs field is a section $\eta\in
H^0(\Hom(W,\mathcal{O})\otimes K)$, that is
$$\eta:W\rightarrow \mathcal{O}\otimes K,$$ and hence $\SO_0(1,n)$-Higgs
bundles $(E,\varphi)$ are in one-to-one correspondence with tuples
$(\mathcal{O},W,Q_W,\eta)$.

%An isomorphism between two $\SO_0(1,n)$-Higgs bundles
%$(\mathcal{O},W,Q_W,\eta)$ and $(\mathcal{O},W',Q_{W'},\eta')$ is
%given by an isomorphism $g:W\rightarrow W'$ such that $\eta'\circ
%g=\eta$.$\newline$

Let $(E,\varphi)$ be an $\SO_0(1,n)$-Higgs bundle. Extending the
structure group of $E$ from
$\SO(1,\mathbb{C})\times\SO(n,\mathbb{C})$ to
$\SO(n+1,\mathbb{C})$, the pair
$(E_{\SO(n+1,\mathbb{C})},\varphi)$, with $$\varphi\in
H^0(E_{\SO(1,\mathbb{C})\times\SO(n,\mathbb{C})}(\mathfrak{m}^\mathbb{C})\otimes
K)\subset
H^0(E_{\SO(n+1,\mathbb{C})}(\mathfrak{so}(n+1,\mathbb{C}))\otimes
K),$$ is an $\SO(n+1,\mathbb{C})$-Higgs bundle.

In terms of vector bundles, if $\textbf{E}$ is the vector bundle
associated to $E_{\SO(n+1,\mathbb{C})}$ via the standard
representation of $\SO(n+1,\mathbb{C})$ in $\mathbb{C}^{n+1}$ and
$(\mathcal{O},W,Q_W,\eta)$ is the tuple corres\-ponding to
$(E,\varphi)$, then $\textbf{E}=\mathcal{O}\oplus W$, and the
$\SO(n+1,\mathbb{C})$-Higgs bundle associated to
$(\mathcal{O},W,Q_W,\eta)$ is the triple
$$(\textbf{E}=\mathcal{O}\oplus
W,Q=\left(%
\begin{array}{cc}
  1 &  \\
   & Q_W \\
\end{array}%
\right),\phi=\left(%
\begin{array}{cc}
   & \eta \\
  -\eta^\top &  \\
\end{array}%
\right)).$$

\section{Stability conditions}

In this section we study the notions of semistability, stability
and polystability for $\SO_0(1,n)$-Higgs bundles, for the
associated $\SO(n+1,\mathbb{C})$-Higgs bundles and the relation
between them. These notions have been studied in \cite{tesis}
applying the general notions given by Bradlow, Garc\'ia-Prada,
Gothen and Mundet i Riera \cite{BG-PM,G-PGM}, that generalize the
results given by Ramanathan \cite{Rt1} for principal bundles.

We will use these notions in term of filtrations. In the case of
$\SO_0(1,n)$-Higgs bundles, since $\SO(1,\mathbb{C})=\{1\}$, they
will only involve conditions on the filtrations of the principal
$\SO(n,\mathbb{C})$-Higgs bundle $(W,Q_W)$.

\begin{definition}\label{stability3}
Let $(\mathcal{O},W,Q_W,\eta)$ be an $\SO_0(1,n)$-Higgs bundle
with $n\neq 2$, then it is \textbf{semistable} if for any
filtration
$$\mathcal{W}=(0\subset W_1\subset\cdots\subset W_s=W),$$ satisfying
$W_j=V_{s-j}^{\bot_{Q_W}}$ and any element
$\mu\in\Lambda(\mathcal{W})$ with
$$\Lambda(\mathcal{W})=\{\mu=(\mu_1,\mu_2,\ldots,\mu_s)\in\mathbb{R}^s\mid\mu_i\leq\mu_{i+1},\mu_{s-i+1}+\mu_i=0
\text{ for any }i\},$$ such that $\eta\in H^0(N\otimes K)$, where
$$N=N(\mathcal{W},\mu)=\sum_{\mu_i\geq 0}\Hom(W_i,\mathcal{O}),$$
we have
$$d(\mathcal{W},\mu)\geq 0.$$ The tuple $(\mathcal{O},W,Q_W,\eta)$ is \textbf{stable} if it is semistable and for any choice of
the filtration $\mathcal{W}$ and non-zero
$\mu\in\Lambda(\mathcal{W})$, such that $\eta\in H^0(N\otimes K)$,
we have
$$d(\mathcal{W},\mu)>0.$$
Finally, the tuple $(\mathcal{O},W,Q_W,\eta)$ is
\textbf{polystable} if it is semistable and for any filtration
$\mathcal{W}$ as above and non-zero $\mu\in\Lambda(\mathcal{W})$
satisfying $\mu_i<\mu_{i+1}$ for each $i$, $\eta\in H^0(N\otimes
K)$ and $d(\mathcal{W},\mu)=0$, there is a splitting
$$W\simeq W_1\oplus W_2/W_1\oplus\cdots\oplus W/W_{s-1}$$
satisfying
$$Q_W(W_i/W_{i-1},W_j/W_{j-1})=0\text{ unless
}i+j=s+1,$$ with respect to which
$$\eta\in
H^0(\bigoplus_{\mu_i=0}\Hom(W_i/W_{i-1},\mathcal{O})\otimes K).$$
\end{definition}

\begin{definition}
The \textbf{moduli space} \textbf{of polystable}
$\SO_0(1,n)$-\textbf{Higgs bundles} is defined as the set of
isomorphisms classes of polystable $\SO_0(1,n)$-Higgs bundles and
is denoted by $\mathcal{M}(\SO_0(1,n))$.
\end{definition}

In the following proposition we prove that the notions of
semistability and stabi\-li\-ty can be simplified.

\begin{proposition}\label{stability4}
Let $(\mathcal{O},W,Q_W,\eta)$ be an $\SO_0(1,n)$-Higgs bundle
with $n\neq 2$. It is \textbf{semistable} if and only if for any
isotropic subbundle $W'\subset W$ such that $\eta(W')=0$ the
inequality $\deg W'\leq 0$ holds. It is \textbf{stable} if and
only if it is semistable and for any non-zero isotropic subbundle
$W'\subset W$ such that $\eta(W')=0$ we have $\deg W'<0$.
\end{proposition}

\begin{proof}
Let $(\mathcal{O},W,Q_W,\eta)$ be an $\SO_0(1,n)$-Higgs bundle
and assume that for any isotropic subbundle $W'\subset W$ such
that $\eta(W')=0$, we have $\deg W'\leq 0$ holds. We want to prove
that $(\mathcal{O},W,Q_W,\eta)$ is semistable.

Choose a filtration $\mathcal{W}=(0\subset W_1\subset\ldots\subset
W_s=W$ satisfying $W_j=W_{s-j}^{\bot_{Q_W}}$ for any $j$. We have
to understand the geometry of the convex set
$$\Lambda=
\{\mu\in\Lambda(\mathcal{W})\mid\eta\in N\}\subset\mathbb{R}^s.$$

Let
$$\mathcal{J}=\{i\mid\eta(W_i)=0\}=\{i_1,\dots,i_k\}.$$ One checks easily that if
$\mu\in\Lambda(\mathcal{W})$, then
$$\mu\in\Lambda\Leftrightarrow \mu_a=\mu_b, \text{ for any $i_l\leq a\leq b\leq
i_{l+1}$.}$$  The set of indices $\mathcal{J}$ is symmetric, that
is
$$i\in\mathcal{J}\Leftrightarrow s-i\in\mathcal{J}.$$
%To check this we have to prove that $\eta(W_i)=0$ implies that
%$\eta(W_i^{\perp_Q})\subseteq \mathcal{O}\otimes K$. Suppose that
%this is not true, then there is a $i$ with $\eta(W_i)=0$ and there
%exists some $w\in W_i^{\perp_{Q_W}}$ such that $\eta(w)\notin
%V_i^{\perp_{Q_V}}\otimes K$. Then there exists $v\in V_i$ such
%that $Q_V(v,\eta(w))\neq 0$. We have
%$$Q_V(v,\eta(w))=Q_W(-\eta^\top(v),w)=0,$$ and the latter vanishes
%because by assumption $-\eta^\top(v)$ belongs to $W_j$. So we have
%reached a contradiction.

Let $\mathcal{J}'=\{i\in\mathcal{J}\mid 2i\leq s\}$ and define for
any $i\in\mathcal{J}'$ the vector
$$L_i=-\sum_{c\leq i}e_c+\sum_{d\geq s-i+1}e_d,$$
where $\{e_1,\ldots,e_s\}$ are the canonical basis of
$\mathbb{R}^s$. The set $\Lambda$ is the positive span of the
vectors $\{L_i\mid i\in\mathcal{J}'\}$ and we have that
$$d(\mathcal{W},\mu)\geq 0\text{ for any $\mu\in\Lambda$}
\Leftrightarrow d(\mathcal{W},L_i)\geq 0\text{ for any $i$.}$$ We
also have that $$d(\mathcal{W},L_i)=-\deg W_{s-i}-\deg W_i.$$
Since $\deg W_{s-i}=\deg W_i$, then $d(\mathcal{W},L_i)=-2\deg
W_i\geq 0$ is equivalent to $\deg W_i\leq 0$, which holds by
assumption. Hence $(\mathcal{O},W,Q_W,\eta)$ is semistable.

Conversely, if $(\mathcal{O},W,Q_W,\eta)$ is semistable, for any
isotropic subbundle $W'\subset W$ such that $\eta(W')=0$ we have
that the condition $\deg W'\leq 0$ is immediately satisfied by
applying the semistability condition of the filtration $0\subset
W'\subset W'^{\perp_{Q_W}}\subset W$.

Finally, the proof of the second statement on stability is very
similar to case of semistability and we then omit it.
\end{proof}

\begin{remark}The case $n=2$ requires special attention. Observe that a
principal $\SO(2,\mathbb{C})$-bundle $(E,Q)$ decomposes as
$E=L\oplus L^{-1}$,
where $L$ is a line bundle and ${\small Q=\left(%
\begin{array}{cc}
   & 1 \\
  1 &  \\
\end{array}%
\right)}$. Then, any principal $\SO(2,\mathbb{C})$-bundle has an
isotropic subbundle with degree greater or equal than zero.
However, $\SO(2,\mathbb{C})\cong\mathbb{C}^*$ has no proper
parabolic subgroups, and the stability condition can not be
simplified in terms of isotropic subbundles.
It seems that this case was overlooked in \cite{Rt1}.
\end{remark}

We now study the relation between the stability of an
$\SO_0(1,n)$-Higgs bundle and the stability of its associated
$\SO(n+1,\mathbb{C})$-Higgs bundle. To do this we introduce the
notions of semistability, stability and polystability for
$\SO(n,\mathbb{C})$-Higgs bundles.

\begin{definition}\label{stability1}
An $\SO(n,\mathbb{C})$-Higgs bundle $(\textbf{E},Q,\phi)$ with
$n\neq 2$ is \textbf{semistable} if for any filtration
$$\mathcal{E}=(0\subset E_1\subset\ldots\subset
E_k=\textbf{E}),$$ $1\leq k\leq n$, satisfying
$E_j=E_{k-j}^{\bot_Q}$, and any element of
$$\Lambda(\mathcal{E})=\{\lambda=(\lambda_1\leq\lambda_2\leq\ldots\leq\lambda_k)\in\mathbb{R}^k\mid
\lambda_{k-i+1}+\lambda_i=0 \text{ for any }i\}$$
%$$\begin{array}{ccl}
%\Lambda(\mathcal{E})&=&\{\lambda=(\lambda_1\leq\lambda_2\leq\ldots\leq\lambda_m)\in\mathbb{R}^m\mid
%\lambda_{m-i+1}+\lambda_i=0 \text{ for any }i\\
%&& \text{ and with exactly $k$ $\lambda_i$ different}\},
%\end{array}$$
such that $\phi\in H^0(N(\mathcal{E},\lambda)\otimes K)$, where
$$N(\mathcal{E},\lambda)=\mathfrak{so}(\textbf{E})\cap\sum_{\lambda_j\leq\lambda_i}\Hom(E_i,E_j),$$
we have
$$d(\mathcal{E},\lambda)=\sum_{j=1}^{k-1}(\lambda_j-\lambda_{j+1})\deg
E_j\geq 0.$$

The triple $(\textbf{E},Q,\phi)$ is \textbf{stable} if it is
semistable and for any choice of the filtration $\mathcal{E}$ and
non-zero $\lambda\in\Lambda(\mathcal{E})$ such that $\phi\in
H^0(N(\mathcal{E},\lambda)\otimes K)$, we have
$$d(\mathcal{E},\lambda)>0.$$

Finally, the triple $(\textbf{E},Q,\phi)$ is \textbf{polystable}
if it is semistable and for any filtration $\mathcal{E}$ as above
and $\lambda\in\Lambda(\mathcal{E})$ satisfying
$\lambda_i<\lambda_{i+1}$ for each $i$, $\phi\in
H^0(N(\mathcal{E},\lambda)\otimes K)$ and
$d(\mathcal{E},\lambda)=0$, there is an isomorphism
$$\textbf{E}\simeq E_1\oplus E_2/E_1\oplus\cdots\oplus
E_k/E_{k-1}$$ satisfying $$Q(E_i/E_{i-1},E_j/E_{j-1})=0\text{
unless }i+j=k+1.$$ Furthermore, via this isomorphism,
$$\phi\in
H^0(\bigoplus_{i}\Hom(E_i/E_{i-1},E_i/E_{i-1})\otimes K).$$
\end{definition}

There is a simplification of the semistability and stability
conditions, which is next described.

\begin{proposition}\label{stability2}
An $\SO(n,\mathbb{C})$-Higgs bundle $(\textbf{E},Q,\phi)$ with
$n\neq 2$ is \textbf{semistable} if and only if for any isotropic
subbundle $E'\subset\textbf{E}$ such that $\phi(E')\subseteq
E'\otimes K$ the inequality $\deg E'\leq 0$ holds, and it is
\textbf{stable} if it is semistable and for any non-zero isotropic
subbundle $E'\subset\textbf{E}$ such that $\phi(E')\subseteq
E'\otimes K$ we have $\deg E'<0$.
\end{proposition}

\begin{proof}
This proof is analogous to the proof of Theorem $3.9$ in
\cite{G-PGM}.

Let $(\textbf{E},Q,\phi)$ be an $\SO(n,\mathbb{C})$-Higgs bundle
and assume that for any isotropic subbundle $E'\subset\textbf{E}$
such that $\phi(E')\subseteq E'\otimes K$ one has $\deg E'\leq 0$.
We are going to prove that $(\textbf{E},Q,\phi)$ is semistable.

Choose any filtration $\mathcal{E}=(0\subset
E_1\subset\ldots\subset E_k=\textbf{E})$ satisfying
$E_j=E_{k-j}^{\bot_Q}$ for any $j$ and consider the set
$$\Lambda(\mathcal{E},\phi)=
\{\lambda\in\Lambda(\mathcal{E})\mid\phi\in
N(\mathcal{E},\lambda)\}\subset\mathbb{R}^k.$$ Let
$\mathcal{J}=\{j\mid\phi(E_j)\subseteq E_j\otimes K\}=
\{j_1,\dots,j_r\}$. One checks easily that if
$\lambda=(\lambda_1,\dots,\lambda_k)\in\Lambda(\mathcal{E})$ then
$$\lambda\in\Lambda(\mathcal{E},\phi)\Leftrightarrow
\lambda_a=\lambda_b\text{ for any $j_i\leq a\leq b\leq
j_{i+1}$}.$$ The set of indices $\mathcal{J}$ is symmetric, i.e.,
$$j\in\mathcal{J}\Leftrightarrow k-j\in\mathcal{J}.$$
To check this we have to prove that $\phi(E_j)\subseteq E_j\otimes
K$ implies that $\phi(E_j^{\perp_Q})\subseteq E_j^{\perp_Q}\otimes
K$. Suppose that this is not true, then there is a $j$ with
$\phi(E_j)\subseteq E_j\otimes K$ and there exists some $w\in
E_j^{\perp_Q}$ such that $\phi(w)\notin E_j^{\perp_Q}\otimes K$.
Then there exists $v\in E_j$ such that $Q(v,\phi(w))\neq 0$.
However, since $\phi\in H^0(\mathfrak{so}(\textbf{E})\otimes K)$,
we must have
$$Q(v,\phi(w))=Q(v,-\phi^\top(w))=-Q(\phi(v),w),$$ and the latter vanishes
because by assumption $\phi(v)$ belongs to $E_j$. So we have
reached a contradiction.

Let $\mathcal{J}'=\{j\in\mathcal{J}\mid 2j\leq k\}$ and define for
any $j\in\mathcal{J}'$ the vector
$$L_j=-\sum_{c\leq j}e_c+\sum_{d\geq k-j+1}e_d,$$
where $e_1,\ldots,e_k$ is the canonical basis of $\mathbb{R}^k$.
We know that the set $\Lambda(\mathcal{E},\phi)$ is the positive
span of the vectors $\{L_j\mid j\in\mathcal{J}'\}$. Consequently,
we have
$$d(\mathcal{E},\lambda)\geq 0\text{ for any $\lambda\in\Lambda(\mathcal{E},\phi)$ }
\Leftrightarrow d(\mathcal{E},L_j)\geq 0\text{ for any $j$ }$$ and
$d(\mathcal{E},L_j)=-\deg E_{k-j}-\deg E_j$. Since $\deg
E_{k-j}=\deg E_j$, $d(\mathcal{E},L_j)\geq 0$ is equivalent to
$\deg E_j\leq 0$, which holds by assumption. Hence
$(\textbf{E},Q,\phi)$ is semistable.

Conversely, if $(\textbf{E},Q,\phi)$ is semistable then for any
isotropic subbundle $E'\subset\textbf{E}$ such that
$\phi(E')\subseteq E'\otimes K$ we have $\deg E'\leq 0$ is
immediate by applying the semistability condition of the
filtration $0\subset E'\subset E'^{\perp_Q}\subset\textbf{E}$.

Finally, the proof of the second statement on stability is very
similar to the case of semistability, we thus omit it.
\end{proof}

\begin{proposition}\label{stabilitycorrespondence}
Let $(\mathcal{O},W,Q_W,\eta)$ be an $\SO_0(1,n)$-Higgs bundle and
let $(\textbf{E},Q,\phi)$ be the corresponding
$\SO(n+1,\mathbb{C})$-Higgs bundle. If $(\mathcal{O},W,Q_W,\eta)$
is stable, then $(\textbf{E},Q,\phi)$ is stable as
$\SO(n+1,\mathbb{C})$-Higgs bundle.
\end{proposition}

\begin{proof}
Let $(\mathcal{O},W,Q_W,\eta)$ be a semistable $\SO_0(1,n)$-Higgs
bundle and consider the associated $\SO(n+1,\mathbb{C})$-Higgs
bundle $(\textbf{E},Q,\phi)$. We will see that for every isotropic
subbundle $E'\subset\textbf{E}$ such that $\phi(E')\subseteq E'$
we have $\deg E'\leq 0$.

If $E'\subset\textbf{E}$ is an isotropic subbundle, we consider
the projection $p:\textbf{E}\rightarrow W$ and the subbundles
$W'=p(E')$ and $V'=E'\cap\mathcal{O}$. Observe that
$V'=\mathcal{O}$ or $0$. We have the exact sequence
$$0\rightarrow V'\rightarrow E'\rightarrow W'\rightarrow 0$$ and
the equality $$\deg E'=\deg W'.$$ Since $Q=\small{\left(%
\begin{array}{cc}
  1 &  \\
   & Q_W \\
\end{array}%
\right)}$, we have \begin{eqnarray*}
  (E')^{\bot_\textbf{E}} &=& (V'\oplus W')^{\bot_\textbf{E}}=(V')^{\bot_\textbf{E}}\cap(W')^{\bot_\textbf{E}} \\
   &=& [(V')^{\bot_\mathcal{O}}\oplus W]\cap[V\oplus(W')^{\bot_W}]=(V')^{\bot_\mathcal{O}}\oplus(W')^{\bot_W},
\end{eqnarray*} and then, the condition
$E'\subseteq(E')^{\bot_\textbf{E}}$ implies
$V'\subseteq(V')^{\bot_\mathcal{O}}$ and
$W'\subseteq(W')^{\bot_W}$, that is, $V'$ and $W'$ are isotropic
subbundles of $\mathcal{O}$ and $W$ respectively. This implies
that $V'=0$. On the other hand, since $\phi(E')\subseteq E'\otimes
K$ and
$\phi=\small{\left(%
\begin{array}{cc}
   & \eta \\
  -\eta^\top &  \\
\end{array}%
\right)}$, we have that $\eta(W')=0$.

The semistability condition for $(\mathcal{O},W,Q_W,\eta)$ gives
$\deg E'=\deg W'\leq 0$ and then we conclude that the
semistability of an $\SO_0(1,n)$-Higgs bundle implies the
semistability of its associated $\SO(n+1,\mathbb{C})$-Higgs
bundle.

Let now $E'\subset\textbf{E}$ be a non-zero isotropic subbundle
such that $\phi(E')\subseteq E'\otimes K$. Since $E'\neq 0$ and it
is isotropic, $W'=p(E')$ is non-zero. The stability condition for
$(\mathcal{O},W,Q_W,\eta)$ gives $\deg E'=\deg W'<0$ and we
conclude.
\end{proof}

\section{Polystable $\SO_0(1,2m+1)$-Higgs bundles}

The main result in this section is Theorem \ref{polystable} which
gives a full description of polystable $\SO_0(1,n)$-Higgs bundles.

For this, we need to describe some special $\SO_0(1,n)$-Higgs
bundles which arise from $G$-Higgs bundles, for certain real
subgroups $G$ of $\SO_0(1,n)$. Consider $\U(n')\subset\SO_0(1,n)$.
From a $\U(n')$-Higgs bundle, that is, a holomorphic vector bundle
$W'$ of rank $n'$, we can obtain an $\SO(2n')$-Higgs bundle
considering the orthogonal bundle $(W'\oplus
W'^*,\langle\cdot,\cdot\rangle)$ where $\langle\cdot,\cdot\rangle$
denotes the dual pairing. This principal
$\SO(2n',\mathbb{C})$-bundle is a special $\SO_0(1,n)$-Higgs
bundle with $V=0$, $(W,Q)=(W'\oplus
W'^*,\langle\cdot,\cdot\rangle)$, $n=2n'$ and $\eta=0$. Consider
now the inclusion $\SO(n')\subset\SO_0(1,n)$. An $\SO(n')$-Higgs
bundle $(W',Q')$ corresponds to an $\SO_0(1,n)$-Higgs bundle with
$V=0$, $(W,Q)=(W',Q')$, $n=n'$ and $\eta=0$.

%We have the following inclusions:
%$$\begin{array}{ccccc}
%  \U(n') & \hookrightarrow & \SO(2n')&\hookrightarrow &\SO_0(1,n) \\
%  W' & \mapsto & (W'\oplus W'^*,\langle\cdot,\cdot\rangle) &
%  \mapsto & (-,W'\oplus W'^*,\langle\cdot,\cdot\rangle,-),
%  \\\\
%  && \SO(n') & \hookrightarrow & \SO_0(1,n)\\
%  &&(W',Q') & \mapsto & (-,W',Q',-),
%\end{array}$$ where $\langle\cdot,\cdot\rangle$ denotes the dual
%pairing.

\begin{theorem}\label{polystable}
Let $(\mathcal{O},W,Q_W,\eta)$ be a polystable $\SO_0(1,n)$-Higgs
bundle. There is a decomposition, unique up to reordering, of this
Higgs bundle as a sum of stable $G_i$-Higgs bundles, where $G_i$
is one of the following groups: $\SO_0(1,n_i)$, $\SO(n_i)$ or
$\U(n_i)$.
\end{theorem}

\begin{proof}
Let $(\mathcal{O},W,Q_W,\eta)$ be a polystable $\SO_0(1,n)$-Higgs
bundle. For $(W,Q_W)$ we fix a filtration $\mathcal{W}=(0\subset
W_1\subset\cdots\subset W_s=W)$ with $W_j=W_{s-j}^{\bot_{Q_W}}$
and a strictly antidominant character $\mu_1<...<\mu_s$ with
$\mu_{s-i+1}+\mu_i=0$, such that $\eta\in
H^0(\displaystyle{\bigoplus_{\mu_i\geq
0}\Hom(W_i,\mathcal{O})\otimes K})$ and $d(\mathcal{W},\mu)=0$.
Since $(\mathcal{O},W,Q_W,\eta)$ is polystable, we have
$$W\simeq W_1\oplus W_2/W_1\oplus\cdots\oplus
W/W_{s-1},$$ with
$$Q_W(W_i/W_{i-1},W_j/W_{j-1})=0\text{ unless }i+j=s+1,$$ and $$\eta\in
H^0(\bigoplus_{\mu_i=0}\Hom(W_i/W_{i-1},\mathcal{O})\otimes K).$$
From the condition $$Q_W(W_i/W_{i-1},W_j/W_{j-1})=0\text{ unless
}i+j=s+1,$$ we have that the bilinear form $Q_W$ gives an
isomorphism $(W_i/W_{i-1})^*\cong W_{s-i+1}/W_{s-i}$. We have the
exact sequence
$$\xymatrix{W_i^\perp\ar[r]& W_{i-1}^\perp\ar[r]^-p &(W_i/W_{i-1})^*}$$ where $p$ is given by
$w\mapsto Q_W(W,\cdot)$, and then $$(W_i/W_{i-1})^*\cong
W_{i-1}^\perp/W_i^\perp\cong W_{s-i+1}/W_{s-i}.$$

Suppose that $n$ is odd and that we have a filtration
$\mathcal{W}=(0\subset W_1\subset...\subset W_s=W)$ where $s$ is
even, then
$W_{\frac{s}{2}}^\bot=W_{s-\frac{s}{2}}=W_{\frac{s}{2}}$. On the
other hand, $\rk(W_{\frac{s}{2}}^\bot)=n-\rk(W_{\frac{s}{2}}),$
that implies $\rk(W_{\frac{s}{2}})=\frac{n}{2}$, which is not a
natural number. Then, if $n$ is odd, all the possible filtrations
$\mathcal{W}=(0\subset W_1\subset...\subset W_s=W)$, have odd
length $s$, and the value $0$ always appears in the middle of the
corresponding strictly antidominant character $\mu_1<...<\mu_s$.
When the rank $n$ is even, we have filtrations for all $1\leq
s\leq n$. When $s$ is odd, we have $\mu_{\frac{s+1}{2}}=0$ and in
the even case, we have
$\cdots\mu_{\frac{s}{2}}<\mu_{\frac{s}{2}+1}\cdots$, with
$\mu_{\frac{s}{2}}=-\mu_{\frac{s}{2}+1}<0$.

The Higgs field can be not equal to zero only when
$\mu_{\frac{s+1}{2}}=0$, that is
$$\eta\in
H^0(\Hom(W_{\frac{s+1}{2}}/W_{\frac{s-1}{2}},\mathcal{O})\otimes
K).$$ Since
$$(W_{\frac{s+1}{2}}/W_{\frac{s-1}{2}})^*\cong
W_{\frac{s+1}{2}}/W_{\frac{s-1}{2}},$$ the tuple
$$(\mathcal{O}\text{, }W_{\frac{s+1}{2}}/W_{\frac{s-1}{2}}\text{, }Q_W\text{, }\eta)$$
is in itself an $\SO_0(1,n_i)$-Higgs bundle, where
$n_i=\rk(W_{\frac{s+1}{2}}/W_{\frac{s-1}{2}})$. Observe that $Q_W$
denotes now the restriction to
$W_{\frac{s+1}{2}}/W_{\frac{s-1}{2}}$.

If $0=\eta\in
H^0(\Hom(W_{\frac{s+1}{2}}/W_{\frac{s-1}{2}},\mathcal{O})\otimes
K),$ $(\mathcal{O},W_{\frac{s+1}{2}}/W_{\frac{s-1}{2}},Q_W,\eta)$
is the sum of the trivial bundle together with an $\SO(n_i)$-Higgs
bundle $(W_{\frac{s+1}{2}}/W_{\frac{s-1}{2}},Q_W)$.

When $\mu_i\neq 0$, we have a pair of $\U(n_i)$-Higgs bundles
$$W_i/W_{i-1}\text{ and }W_{s-i+1}/W_{s-i},$$ dual one to the
other. In this case $n_i=\rk(W_i/W_{i-1})=\rk(W_{s-i+1}/W_{s-i})$.

Each piece in the decomposition is also polystable, and we can
repeat the process and obtain a decomposition where all the pieces
are stable Higgs bundles (using the Jordan-H\"older reduction,
\cite[Sec. $2.10$]{G-PGM}).
\end{proof}

Observe that there can only be one summand with $G_i=\SO_0(1,n_i)$
in the decomposition.

If in the decomposition of a polystable $\SO_0(1,n)$-Higgs bundle
$(\mathcal{O},W,Q_W,\eta)$ there is a summand which is an
$\SO(2)$-Higgs bundle, that is, a principal
$\SO(2,\mathbb{C})$-bundle $E=L\oplus L^{-1}$, the isotropic
subbundles $L$ and $L^{-1}$, which have opposite degrees, do not
violate the stability condition for $E$ (since there are no
parabolic subgroups in $\SO(2,\mathbb{C})$) but they violate the
stability condition for $(\mathcal{O},W,Q_W,\eta)$.

Equivalently, if there is a summand in the decomposition which is
a $\U(n_i)$-Higgs bundles $E$, then $E^*$ is also in the
decomposition of $(\mathcal{O},W,Q_W,\eta)$, and since
$\deg(E)=-\deg(E)$, one or both vector bundles violate the
stability condition for $(\mathcal{O},W,Q_W,\eta)$.

Then we have the following results.

\begin{proposition}\label{stable}
If a polystable $\SO_0(1,n)$-Higgs bundle
$(\mathcal{O},W,Q_W,\eta)$ decomposes as a sum of stable
$G_i$-Higgs bundles where $G_i=\SO_0(1,n_i)$ and $\SO(n_i)$ with
$n_i\neq 2$, then $(\mathcal{O},W,Q_W,\eta)$ is stable.
\end{proposition}

\begin{proposition}\label{strictlypolystable}
If an $\SO_0(1,n)$-Higgs bundle $(\mathcal{O},W,Q_W,\eta)$ is
strictly polystable, then in its decomposition there must be at
least a $G_i$-Higgs bundle with $G_i=\U(n_i)$ or $\SO(2)$.
\end{proposition}

Theorem \ref{polystable} gives us a decomposition of a polystable
$\SO_0(1,n)$-Higgs bundles as a sum of stable $G_i$-Higgs bundles,
where $G_i$ is one of the following groups: $\SO_0(1,n_i)$,
$\SO(n_i)$ or $\U(n_i)$. From the following result we have that,
in fact, any polystable $\SO_0(1,n)$-Higgs bundles can be
decomposed as a sum of smooth $G_i$-Higgs bundles.

\begin{proposition}\label{polystablessmooth}
Let $(\mathcal{O},W,Q_W,\eta)$ be a polystable $\SO_0(1,n)$-Higgs
bundle. There is a decomposition, unique up to reordering, of this
Higgs bundle in a sum of smooth $G_i$-Higgs bundles, where
$G_i=\SO_0(1,n_i)$, $\SO(n_i)$ or $\U(n_i)$.
\end{proposition}

\begin{proof}
The starting point is Theorem \ref{polystable}.

A stable $\U(n)$-Higgs bundle represents a smooth point in the
moduli space of $\U(n)$-Higgs bundles.

A stable $\SO(n)$-Higgs bundle is smooth if and only if it is
stable and simple. On the other hand, any stable $\SO(n)$-Higgs
bundle which is not simple can be expressed, using Theorem
\ref{stablenonsimpleSO(n)}, as a direct sum of smooth
$\SO(n_i)$-Higgs bundles.

Finally, as we know from Corollary \ref{smooth}, a stable
$\SO_0(1,n)$-Higgs bundle represents a smooth point of the moduli
space if and only if it is simple, but if a stable
$\SO_0(1,n)$-Higgs bundle is non-simple, from Theorem
\ref{stablenonsimple} we have that it decomposes as a sum of
smooth $\SO_0(1,n_i)$ and $\SO(n_i)$-Higgs bundles.
\end{proof}

\section{Smoothness and deformation
theory}\label{deformationsection}

It is known that a stable vector bundle is simple and that it is a
smooth point of the moduli space of polystable vector bundles. On
the other hand, a stable principal $\SO(n,\mathbb{C})$-bundle with
$n\neq 2$ represents a smooth point of the moduli space
$\mathcal{M}(\SO(n))$ if and only if it is simple (see \cite{Ra}).
Observe that, for $n=2$, we have
$\SO(2,\mathbb{C})\cong\mathbb{C}^*$ and then any $\SO(2)$-Higgs
bundle is stable, simple and smooth. Thus, except in the case
$n=2$, the stability of a special orthogonal bundle does not imply
simplicity. In this section we study the smoothness conditions in
the moduli space $\mathcal{M}(\SO_0(1,n))$ adapting the results in
\cite[Sec. $4.2$]{G-PGM} to our case.

\begin{definition}
A $G$-Higgs bundle $(E,\varphi)$ is said to be \textbf{simple} if
$\Aut(E,\varphi)=\ker\iota\cap Z(H^\mathbb{C})$, where $H\subset
G$ is a maximal compact subgroup, $Z(H^\mathbb{C})$ denotes the
centre of its complexification and
$\iota:H^\mathbb{C}\rightarrow\GL(\mathfrak{m}^\mathbb{C})$ is the
complexified isotropy representation corresponding to the Cartan
decomposition $\mathfrak{g}=\mathfrak{h}+\mathfrak{m}$ of the Lie
algebra of $G$.
\end{definition}

A $G$-Higgs bundle is then simple if the group of automorphisms is
as small as possible. To be in $\ker\iota$ means to be compatible
with the Higgs field.

If $(E,Q)$ is an $\SO(n)$-Higgs bundle with $n>2$, that is, a
principal $\SO(n,\mathbb{C})$-bundle, it has Higgs field equal to
zero and then it is simple if and only if
$$\Aut(E,Q)=Z(\SO(n,\mathbb{C}))=\left\{%
\begin{array}{ll}
    I_n, & \hbox{$n$ odd,} \\
    \pm I_n, & \hbox{$n$ even.} \\
\end{array}%
\right.$$ The group of automorphisms of an $\SO_0(1,n)$-Higgs
bundle is
$$\Aut(\mathcal{O},W,Q_W,\eta)=\{(1,g)\in\Aut(\mathcal{O})\times\Aut(W,Q_W)\mid\eta\circ
g=\eta\},$$ and hence $(\mathcal{O},W,Q_W,\eta)$ is simple if and
only if
$$\Aut(\mathcal{O},W,Q_W,\eta)=\ker\iota\cap
\{1\}\times Z(\SO(n,\mathbb{C}))=\{I_{n+1}\}.$$

Let us consider the \textbf{deformation complex} of an
$\SO_0(1,n)$-Higgs bundle
$$\begin{array}{rcl}
C^\bullet(\mathcal{O},W,Q_W,\eta):
\mathfrak{so}(\mathcal{O})\oplus\mathfrak{so}(W)&\rightarrow&
\Hom(W,\mathcal{O})\otimes
K,\\
 (0,g)\text{ }&\mapsto&\eta g.
\end{array}$$ (see \cite[Definition $4.3$]{G-PGM}).
We have the following result.

\begin{proposition}\label{hypercohomology}
If $(\mathcal{O},W,Q_W,\eta)$ is an $\SO_0(1,n)$-Higgs bundle, we
have the following:
\begin{enumerate}
    \item The space of endomorphisms of $(\mathcal{O},W,Q_W,\eta)$ is
    isomorphic to the hypercohomology group
    $\mathbb{H}^0(C^\bullet(\mathcal{O},W,Q_W,\eta))$.
    \item The space of infinitesimal deformations of $(\mathcal{O},W,Q_W,\eta)$
    is isomorphic to the first hypercohomology group
    $\mathbb{H}^1(C^\bullet(\mathcal{O},W,Q_W,\eta))$.
%    \item There is a long exact sequence
%    $$0\rightarrow\mathbb{H}^0(C^\bullet(E,\varphi))\rightarrow
%H^0(E(\mathfrak{h}^\mathbb{C}))\rightarrow
%    H^0(E(\mathfrak{m}^\mathbb{C})\otimes
%K)\rightarrow\mathbb{H}^1(C^\bullet(E,\varphi))\rightarrow$$
%    $$\rightarrow H^1(E(\mathfrak{h}^\mathbb{C}))\rightarrow
%H^1(E(\mathfrak{m}^\mathbb{C})\otimes
%    K)\rightarrow\mathbb{H}^2(C^\bullet(E,\varphi))\rightarrow 0,$$ where
%    $H^i(E(\mathfrak{h}^\mathbb{C}))\rightarrow
%H^i(E(\mathfrak{m}^\mathbb{C})\otimes
%    K)$ is induced by $\ad(\varphi)$.
\end{enumerate}
\end{proposition}

It follows from Proposition \ref{hypercohomology} that, for every
$\SO_0(1,n)$-Higgs bundle $(\mathcal{O},W,Q_W,\eta)$
re\-pre\-senting a smooth point of the moduli space, the tangent
space at this point is canonically isomorphic to
$\mathbb{H}^1(C^\bullet(\mathcal{O},W,Q_W,\eta))$.

\begin{proposition}\label{smoothness}
If an $\SO_0(1,n)$-Higgs bundle $(\mathcal{O},W,Q_W,\eta)$ is
stable, simple and satisfies
$$\mathbb{H}^2(C^{\bullet}(\mathcal{O},W,Q_W,\eta))=0,$$ then it is a smooth
point of the moduli space.
\end{proposition}

A $G$-Higgs bundle $(E,\varphi)$ is \textbf{infinitesimally
simple} if
$\End(E,\varphi)\cong\mathbb{H}^0(C^\bullet(E,\varphi))$ is
isomorphic to $H^0(E(\ker d\iota\cap\mathfrak{z}))$. Stable
implies infinitesimally simple.

Let $(\mathcal{O},W,Q_W,\eta)$ be an $\SO_0(1,n)$-Higgs bundle and
consider the associated $\SO(n+1,\mathbb{C})$-Higgs bundle
$(\textbf{E},Q,\phi)$ and the deformation complex
$$\xymatrix{C^\bullet(\textbf{E},Q,\phi):
\mathfrak{so}(\textbf{E})\ar[r]^-{\ad(\varphi)}&
\mathfrak{so}(\textbf{E})\otimes K.}$$ Since $\SO(n+1,\mathbb{C})$
is complex, infinitesimally simple in this case means
$\mathbb{H}^0(C^\bullet(\textbf{E},Q,\phi))=0$ ($\ker
d\iota=\ker(\ad)=0$) and, as in the real case, stable implies
infinitesimally simple. There is an isomorphism
$$\mathbb{H}^2(C^\bullet(\textbf{E},Q,\phi))=\mathbb{H}^0(C^\bullet(\textbf{E},Q,\phi))^*,$$
and we have the following relation
$$\mathbb{H}^0(C^{\bullet}(\textbf{E},Q,\phi))\cong\mathbb{H}^0(C^{\bullet}(E,\varphi))\oplus
\mathbb{H}^2(C^{\bullet}(E,\varphi))^*.$$ Then, if
$(\textbf{E},Q,\phi)$ is stable,
$\mathbb{H}^0(C^\bullet(\textbf{E},Q,\phi))=0$ and this implies
$$\mathbb{H}^0(C^{\bullet}(E,\varphi))=
\mathbb{H}^2(C^{\bullet}(E,\varphi))=0.$$

Using Proposition \ref{stabilitycorrespondence} we obtain the
following description.

\begin{corollary}\label{smooth}
If an $\SO_0(1,n)$-Higgs bundle $(\mathcal{O},W,Q_W,\eta)$ is
stable and simple, then it is a smooth point of the moduli space.
\end{corollary}

\begin{corollary}\label{H0yH2}
Let $(\mathcal{O},W,Q_W,\eta)$ be a stable $\SO_0(1,n)$-Higgs
bundle which re\-pre\-sents a smooth point of the moduli space,
then
$$\mathbb{H}^0(C^\bullet(\mathcal{O},W,Q_W,\eta))=\mathbb{H}^2(C^\bullet(\mathcal{O},W,Q_W,\eta))=0.$$
\end{corollary}

The \textbf{expected dimension} of the moduli space
$\mathcal{M}(\SO_0(1,n))$ (see \cite{G-PGM}), is
$$\dim\mathbb{H}^1(C^\bullet(\mathcal{O},W,Q_W,\eta))=-\chi(C^\bullet(\mathcal{O},W,Q_W,\eta))=
\frac{n(n+1)(g-1)}{2},$$ where
$\dim(\SO_0(1,n))=\frac{n(n+1)}{2}$.

\section{Stable and non-simple $\SO(n)$ and $\SO_0(1,n)$-Higgs
bundles}

In this section we give a description of the stable $\SO(n)$ and
$\SO_0(1,n)$-Higgs bundles that fail to be simple.

\begin{lemma}\label{summandnonstable0}
If an $\SO(n)$-Higgs bundle $(E,Q)$ decomposes as a sum of
$G_i$-Higgs bundles and one of them is an $\SO(n_i)$-Higgs bundle
with $n_i>2$ which is not stable or an $\SO(2)$-Higgs bundle, then
$(E,Q)$ is not stable.
\end{lemma}

\begin{proof}
If there is a summand which is an $\SO(2)$-Higgs bundle
$E_i=L\oplus L^{-1}$, the isotropic subbundles $L$ and $L^{-1}$,
which have opposite degrees, do not violate the stability
condition for $E$ but they violate the stability condition for
$(E,Q)$. If a summand $(E_i,Q_i)$ is a non-stable $\SO(n_i)$-Higgs
bundle, there is a proper isotropic subbundle $F_i\subset E_i$
such that $\deg F_i\geq 0$. Since $Q_i$ is the restriction of $Q$
to $E_i$, $F_i$ is an isotropic subbundle of $E$ that violates its
stability.
\end{proof}

\begin{theorem}\label{stablenonsimpleSO(n)}
Let $(E,Q)$ be a stable $\SO(n)$-Higgs bundle with $n\neq 2$, that
is, a principal $\SO(n,\mathbb{C})$-bundle, which is not simple,
then it decomposes as a sum of stable and simple $\SO(n_i)$-Higgs
bundles with $n_i\neq 2$. Moreover, in the decomposition there
must be at least an $\SO(n_i)$-Higgs bundle with $n_i$ odd.
\end{theorem}

\begin{proof}
Since $(E,Q)$ is not simple and $$Z(\SO(n,\mathbb{C}))=\left\{%
\begin{array}{ll}
    I_n, & \hbox{$n$ odd,} \\
    \pm I_n, & \hbox{$n$ even,} \\
\end{array}%
\right.$$ there is an automorphism $f\in\Aut(E,Q)\backslash\{\pm
I_n\}$ if $n$ even, or $f\in\Aut(E,\varphi)\backslash\{I_n\}$ if
$n$ is odd.

Suppose that $f=\lambda I_n$ with $\lambda\in\mathbb{C}^*$. It has
to preserve the orthogonal structure of $E$, that is,
$$Q(f(e),f(e'))=\lambda^2 Q(e,e')=Q(e,e'),$$ and this happens if
and only if $\lambda=\pm 1$. On the other hand, the determinant of
$f$ has to be equal to one. Then, the only possibilities are
$f=\pm I_n$ if $n$ is even and $f=I_n$ if $n$ is odd, which are
exactly the cases that we are excluding.

The group $\Aut(E,\varphi)$ is reductive. This implies that $f$
may be chosen in such a way that there is a splitting $E=\bigoplus
E_i$ such that $f$ restricted to $E_i$ is $\lambda_i I_n$ with
$\lambda_i\in\mathbb{C}^*$.

Since
$$Q(e_i,e_j)=Q(f(e_i),f(e_j))=\lambda_i\lambda_j
Q(e_i,e_j),$$ then $Q(E_i,E_j)$ can only be non-zero when
$\lambda_i\lambda_j=1$. Since $Q$ is non-degenerate, the possible
values of lambda come in pairs $(\lambda_i,\lambda_i^{-1})$
corresponding to $(E_i,E_i^*)$. If $\lambda_i=\pm 1$, we have
$\lambda_i=\lambda_i^{-1}$ and then $E_1\cong E_1^*$ and
$E_{-1}\cong E_{-1}^*$. Since $\det f=\prod_{i}\lambda_i^{\rk
E_i}=1$, we do not have the value $\lambda_i=0$.

Suppose that there is a $\lambda_i\neq\pm 1$, then $E_i\subset E$
is an isotropic subbundle of $E$. If $\deg E_i\geq 0$, this
subbundle violates the stability condition for $(E,Q)$. If $\deg
E_i<0$, then $\deg E_i^*>0$ and again $(E,Q)$ is not stable. Hence
$\lambda_i=\pm 1$ and $(E,Q)=(E_1,Q_1)\oplus(E_{-1},Q_{-1})$.

From Lemma \ref{summandnonstable0} we have that these summands are
stable $\SO(n_i)$-Higgs bundles with $n_i\neq 2$.

If there is a summand which is a non-simple $\SO(n_i)$-Higgs
bundle, applying the argument of this proof inductively we
conclude that a stable but non-simple $\SO(n)$-Higgs bundle can be
decomposed as a sum of smooth $\SO(n_i)$-Higgs bundles.

Finally, since $(E,Q)$ is not simple, there must be at least an
$\SO(n_i)$-Higgs bundle with $n_i$ even in the decomposition. This
condition allows us to take the automorphism $-1$ in this summand
and guarantee the non-simplicity.
\end{proof}

\begin{lemma}\label{summandnonstable1}
If an $\SO_0(1,n)$-Higgs bundle $(\mathcal{O},W,Q_W,\eta)$
decomposes as a sum of $G_i$-Higgs bundles and one of them is an
$\SO_0(1,n_i)$-Higgs bundle $(\mathcal{O},W_i,Q_W,\eta_i)$ which
is not stable, then $(\mathcal{O},W,Q_W,\eta)$ is not stable.
\end{lemma}

\begin{proof}
Since $(\mathcal{O},W_i,Q_W,\eta_i)$ is not stable, there is an
isotropic subbundle $W'\subset W_i$ (such that
$\eta_i(W')\subseteq \mathcal{O}\otimes K$) with $\deg W'\geq 0$.
But $W'$ is also an isotropic subbundle of $W$ and violates the
stability condition for $(\mathcal{O},W,Q_W,\eta)$.
\end{proof}

\begin{lemma}\label{summandnonstable2}
If an $\SO_0(1,n)$-Higgs bundle $(\mathcal{O},W,Q_W,\eta)$
decomposes as a sum of $G_i$-Higgs bundles and one of them is an
$\SO(2)$-Higgs bundle or an $\SO(n_i)$-Higgs bundle which is not
stable, then $(\mathcal{O},W,Q_W,\eta)$ is not stable.
\end{lemma}

\begin{proof}
It can be deduced from the proof of Lemma \ref{summandnonstable0}
and Lemma \ref{summandnonstable1}.
\end{proof}

\begin{theorem}\label{stablenonsimple}
Let $(\mathcal{O},W,Q_W,\eta)$ be a stable $\SO_0(1,n)$-Higgs
bundle which is not simple, then it decomposes as a sum of stable
and simple $\SO_0(1,n_i)$-Higgs bundles and stable and simple
$\SO(n_i)$-Higgs bundles with $n_i\neq 2$. Moreover, in the
decomposition there must be at least an $\SO(n_i)$-Higgs bundle
with $n_i$ even.
\end{theorem}

\begin{proof}
Suppose that the Higgs field is equal to zero, then the
$\SO_0(1,n)$-Higgs bundle $(\mathcal{O},W,Q_W,\eta)$ is the sum of
the trivial bundle together with a stable principal
$\SO(n,\mathbb{C})$-bundle $(W,Q_W)$, that is, a stable
$\SO(n)$-Higgs bundle. If $(W,Q_W)$ is simple, then we have the
result. If it is not, we conclude using Theorem
\ref{stablenonsimpleSO(n)}.

Suppose now that $\eta\neq 0$. Since $(\mathcal{O},W,Q_W,\eta)$ is
not simple, there is an automorphism
$f\in\Aut(\mathcal{O},W,Q_W,\eta)\backslash\{I\}$. If
$f=(f_1,f_2)$, since $f_1\in\Aut(\mathcal{O})$, we have $f_1=1$.

Suppose that $f=(f_1,f_2)=(1,\mu I)$ is a multiple of the identity
in $W$ ($\mu\in\mathbb{C}^*$). The determinant of $f_2$ has to be
equal to $1$ and $f_2$ has to preserve the orthogonal structure,
that is,
$$Q_W(f_2(w),f_2(w'))=\mu^2 Q_W(w,w')=Q_W(w,w').$$
On the other hand, since we are supposing that $f_2$ is a multiple
of the identity, the condition $f_1\circ\eta=\eta\circ f_2$ is
equivalent to $f_1=f_2$, that is $f=I$, which is exactly the case
that we are excluding. Thus, $f$ is not of this form.

Since the group $\Aut(W,Q_W)$ is reductive, there is a splitting
$W=\bigoplus W_i$ such that $f_2=\mu_i I$ in $W_i$
($\mu_i\in\mathbb{C}^*$). Since
$$Q_W(w_i,w_j)=Q_W(f_2(w_i),f_2(w_j))=\mu_i\mu_j
Q_W(w_i,w_j),$$ then $Q_W(W_i,W_j)$ can only be non-zero when
$\mu_i\mu_j=1$. Since $Q_W$ is non-degenerate, the possible values
of mu come in pairs $(\mu_i,\mu_i^{-1})$ corresponding to
$(W_i,W_i^*)$. If $\mu_i=\pm 1$, we have $\mu_i=\mu_i^{-1}$ and
then $W_1\cong W_1^*$ and $W_{-1}\cong W_{-1}^*$. Since $\det
f_2=\prod_{i}\mu_i^{\rk W_i}=1$, we do not have $\mu_i=0$.

Since $f$ preserve the Higgs field, for each component $\eta_i\in
H^0(\Hom(W_i,\mathcal{O})\otimes K)$, we have that
$$\eta_i(f_2(w))=\mu_i\eta_i(w)$$ is equal to $$f_1(\eta_i(w))=\eta_i(w),$$ for all $w\in
W_i$, and then, $\mu_i\neq 1$ implies $\eta_i=0$.

Suppose that there is a $\mu_i\neq\pm 1$. Then, in particular,
$\mu_i\neq 1$ and we have $\eta_i=0$, that is, $\eta(W_i)=0$.
Since
$$Q_W(W_i,W_i)=Q_W(f_2(W_i),f_2(W_i))=\mu_i^2
Q_W(W_i,W_i),$$ and $\mu_i^2\neq 1$, we have $Q_W(W_i,W_i)=0$ and
hence, $W_i\subset W$ is an isotropic subbundle. If $\deg W_i\geq
0$, this subbundle violates the stability condition for
$(\mathcal{O},W,Q_W,\eta)$. If $\deg W_i<0$, then $\deg W_i^*>0$
and again $(\mathcal{O},W,Q_W,\eta)$ is not stable and we get a
contradiction. Then $\mu_i=\pm 1$.

Since $1=\det f_2=1^{\rk W_1}\cdot(-1)^{\rk W_{-1}}$ we have $\rk
W_{-1}$ even.

We have the following decomposition
$$(\mathcal{O},W,Q_W,\eta)=(\mathcal{O},W_1,\eta_1)\oplus W_{-1}.$$
Since $f_2$ is not a multiple of the identity, $W_1$ and $W_{-1}$
are non-zero, and since $\eta\neq 0$, then $\eta_1\neq 0$. Thus,
$(\mathcal{O},W,Q_W,\eta)$ is a sum of a $\SO_0(1,n_i)$-Higgs
bundle $(\mathcal{O},W_1,\eta_1)$ together with an
$\SO(n_i)$-Higgs bundle $W_{-1}$.

From Lemma \ref{summandnonstable1} and Lemma
\ref{summandnonstable2} we have that these summands are stable
$G_i$-Higgs bundles ($\SO(n_i)$ with $n_i\neq 2$).

If $W_{-1}$ is non-simple, we have from Theorem
\ref{stablenonsimpleSO(n)} that it decomposes as a sum of stable
and simple orthogonal bundles. If $(\mathcal{O},W_1,\eta_1)$ is a
non-simple $\SO_0(1,n_i)$-Higgs bundle, applying the argument of
this proof inductively we conclude that it can be decomposed as a
sum of stable and simple $G_i$-Higgs bundles with
$G_i=\SO_0(1,n_i)$ and $\SO(n_i)$.

Since all the summands are simple and $(\mathcal{O},W,Q_W,\eta)$
is not simple, it must have at least one summand of this type: a
smooth $\SO(n_i)$-Higgs bundle with $n_i$ even. This condition
allow us to take the automorphism $-1$ in this summand and
guarantee the non-simplicity.
\end{proof}

\section{Topology of the moduli spaces}

Let $(\mathcal{O},W,Q_W,\eta)$ be an $\SO_0(1,n)$-Higgs bundle. We
have a topological invariant $c$ associated to it, which is given
by the following exact sequence
$$1\rightarrow\pi_1(\SO(n,\mathbb{C}))\rightarrow
\widetilde{\SO}(n,\mathbb{C})\rightarrow\SO(n,\mathbb{C})\rightarrow
1,$$ where $\widetilde{\SO}(n,\mathbb{C})$ is the universal cover
of $\SO(n,\mathbb{C})$ and the associated long cohomology sequence
$$\xymatrix{H^1(X,\widetilde{\SO}(n,\mathbb{C}))\ar[r]&H^1(X,\SO(n,\mathbb{C}))\ar[r]^-{c}
&H^2(X,\pi_1(\SO(n,\mathbb{C}))).}$$ This invariant
$$c\in H^2(X,\pi_1(\SO(n,\mathbb{C})))\cong\pi_1(\SO(n,\mathbb{C}))$$
measures the obstruction to lifting $(W,Q_W)$ to a flat
$\widetilde{\SO}(n,\mathbb{C}))$-bundle. Observe that when $n\geq
3$, the universal cover of $\SO(n,\mathbb{C})$ is
$\Spin(n,\mathbb{C})$.
We have that $$\pi_1(\SO(n,\mathbb{C}))=\left\{%
\begin{array}{ll}
    1, & \hbox{$n=1$,} \\
    \mathbb{Z}, & \hbox{$n=2$,} \\
    \mathbb{Z}/2, & \hbox{$n\geq 3$.} \\
\end{array}%
\right.$$ When $n\geq 3$, the invariant $c\in\mathbb{Z}/2$
corresponds to the second Stiefel-Whitney classe of the orthogonal
bundle that we obtain from the reduction of the structure group of
$(W,Q_W)$ from $\SO(n,\mathbb{C})$ to the real group $\SO(n)$.

Since $\det W=\mathcal{O}$, using the application
$$\xymatrix{H^1(X,\SO(n,\mathbb{C}))\ar[r]^-{\det}&J(X)}$$ in the Jacobian of $X$
and the identification
$$H^1(X,\mathbb{Z}_2)\cong J_2(X)=\{L\in J(X)\mid
L^2\cong\mathcal{O}\},$$ the first Stiefel-Whitney classes of the
bundle is zero.

We define the moduli space of polystable $\SO_0(1,n)$- Higgs
bundles with invariant $c$ as
$$\mathcal{M}_c(\SO_0(1,n))=\{(\mathcal{O},W,Q_W,\eta)\in\mathcal{M}(\SO_0(1,n))\text{
such that }c(W,Q_W)=c\}.$$

The invariant $c$ gives a first decomposition of the moduli space
$$\mathcal{M}(\SO_0(1,n))=\coprod_c\mathcal{M}_c(\SO_0(1,n)).$$ To
obtain the number of connected components it is necessary to
distinguish which of these components $\mathcal{M}_c(\SO_0(1,n))$
are connected and which decompose as a union of connected
components.

\section{Hitchin fuction}

To simplify, we denote $\mathcal{M}:=\mathcal{M}_c(\SO_0(1,n))$.
Morse-theoretic techniques for studying the topology of moduli
spaces of Higgs bundles were introduced by Hitchin \cite{Hi1,Hi2}.
In this section we describe briefly Hitchin's method and we begin
the study of our particular case.

The moduli space of equivalence classes of reductive
representations in a Lie group $G$ is homeomorphic to the moduli
space of polystable $G$-Higgs bundles. The proof of this result
involves the moduli space of solutions to the \emph{Hitchin's
equations}. It was proved by Hitchin \cite{Hi1} and by Simpson
\cite{S2} for a complex Lie group and by Bradlow, Garc\'ia-Prada,
Gothen and Mundet i Riera \cite{BG-PM,G-PGM} in the real case,
that $\mathcal{M}(G)$ is homeomorphic to the moduli space of
solutions to the Hitchin's equations, $\mathcal{M}^{Hit}(G)$,
which is defined as the space of pairs $(A,\varphi)$, where $A$ is
a connection on a smooth principal $H$-bundle $E_H$ and
$\varphi\in\Omega^{1,0}(E_H(\mathfrak{m}^\mathbb{C}))$, satisfying
$$\begin{array}{rll}
    F_A-[\varphi,\tau(\varphi)] &=& 0, \\
  \bar{\partial}_A(\varphi) &=& 0,
\end{array}$$ modulo gauge equivalence.

Using the homeomorphism
$\mathcal{M}^{Hit}_c(\SO_0(1,n))\cong\mathcal{M},$ the
\textbf{Hitchin function} is defined as the positive function
$$f:\mathcal{M}\rightarrow\mathbb{R},$$ given by
$$[A,\varphi]\mapsto\|\varphi\|^2=\int_X|\varphi|^2 d\vol,$$ where
$[\cdot,\cdot]$ denotes the equivalence class in the moduli space
$\mathcal{M}^{Hit}_c(\SO_0(1,n))$ and $|\cdot|$ is the harmonic
metric that gives the reduction to $\SO(1)\times\SO(n)$.
Equivalently, we can define the map over the moduli space of Higgs
pairs, for a fixed $(E,\varphi)\in\mathcal{M}$, by using the
$L^2$-norm $\|\cdot\|$ of the metric that solves the Hitchin's
equations.

\begin{proposition}\label{fproper}
The function $f([A,\varphi])=\|\varphi\|^2$ is a proper map.
\end{proposition}

The proof of this result was given by Hitchin in \cite[Proposition
$7.1$]{Hi1}.

%Consider a compact subspace $[0,k]\subset\mathbb{R}$ and the
%inverse image
%$$f^{-1}[0,k]=\{[A,\varphi]\mid f([A,\varphi])=\|\varphi\|^2\leq
%k\}.$$ A pair $[A,\varphi]\in f^{-1}[0,k]$ is a solution of the
%Hitchin's equations, that is, $$F_A-[\varphi,\tau(\varphi)]=0.$$
%Since $\|[\varphi,\tau(\varphi)]\|^2$ is bounded by a multiple of
%$\|\varphi\|^2$, we have an $L^2$ bound of the curvature $F_A$,
%and using Uhlenbeck's weak compactness theorem, any infinite
%sequence in $f^{-1}[0,k]$ has convergent subsequence, and hence
%$f^{-1}[0,k]$ is compact.

Even if $\mathcal{M}$ is not smooth, as in our case, the fact that
$f$ is a proper map gives information about the connected
components of $\mathcal{M}$.

\begin{proposition}
Let $\mathcal{M}'\subseteq\mathcal{M}$ be a closed subspace and
let $\mathcal{N}'\subseteq\mathcal{M}'$ be the subspace of local
minima of $f$ on $\mathcal{M}'$. If $\mathcal{N}'$ is connected,
then $\mathcal{M}'$ is connected.
\end{proposition}

This result is in fact more general. The proper function $f$ has a
minimum on each connected component of $\mathcal{M}'$, and then
the number of connected components of $\mathcal{M}'$ is bounded by
the number of connected components of $\mathcal{N}'$. Thus, we are
interested in computing the critical points and more precisely the
local minima of $f$.$\newline$

To study the critical points of the Hitchin function we use the
following results (see \cite{Hi1}).

\begin{proposition}\label{hamiltonian}
The restriction of $f([A,\varphi])=\|\varphi\|^2$ to the smooth
locus $\mathcal{M}^s\in\mathcal{M}$ is a moment map for the
Hamiltonian circle action
$$[A,\varphi]\mapsto[A,e^{i\theta}\varphi].$$
\end{proposition}

\begin{proposition}\label{fixedpoints}
A smooth point of the moduli space $\mathcal{M}$ is a critical
point of $f$ if and only if it is a fixed point of the circle
action, and the subbundle $\nu^-(\mathcal{M}_l)$ where the Hessian
of the Hitchin function is negative definite equals the subbundle
of $\nu(\mathcal{M}_l)$ on which the circle acts with negative
weights.
\end{proposition}

Using Proposition \ref{fixedpoints}, the critical points of $f$
are of two types:

(1) The Higgs field $\varphi=0$.

(2) If $\varphi\neq 0$, $[A,\varphi]$ is a fixed point of the
circle action if and only if
$$[A,e^{i\theta}\varphi]=[A,\varphi]\text{, for all }e^{i\theta}\in S^1.$$
Then, there is a 1-parameter family of gauge transformations
$g(\theta)=(g_1(\theta),g_2(\theta))$ such that
\begin{equation}\label{equation}
    (A,e^{i\theta}\varphi)=g(\theta)\cdot(A,\varphi)=(g(\theta)\cdot
A,g(\theta)\cdot\varphi).
\end{equation}

If the family $\{g(\theta)=(g_1(\theta),g_2(\theta))\}$ is
generated by an infinitesimal gauge transformation
$\psi=(\psi_1,\psi_2)$, we have that
$$g(\theta)\cdot\varphi=\iota(g(\theta))(\varphi)
=\Ad(g(\theta))(\varphi)=\exp(\ad(\theta\psi))(\varphi),$$ and
taking $\frac{d}{d\theta}|_{\theta=0}$ in the second term of the
brackets in (\ref{equation}) we obtain
$$\frac{d}{d\theta}(e^{i\theta}\varphi)|_{\theta=0}=i\varphi,$$
and
$$\frac{d}{d\theta}(g(\theta)\cdot\varphi)|_{\theta=0}=
\frac{d}{d\theta}\exp(\ad(\theta\psi))(\varphi)|_{\theta=0}=\ad(\psi)(\varphi)=[\psi,\varphi].$$
Then $$[\psi,\varphi]=i\varphi.$$

Let $A=(A_1,A_2)$. Since $g_1(\theta)$ and $g_2(\theta)$ act on
$A_1$ and $A_2$ separately, we can consider $\psi_1$ and $\psi_2$
generating the action of $\{g_1(\theta)\}$ and $\{g_2(\theta)\}$.
%$$(\psi_i)_{\omega_i}h=\lim_{\theta\rightarrow
%0}\frac{h(g_i(\theta)\cdot\omega_i)-h(\omega_i)}{\theta},$$ for
%every smooth function $h:\mathcal{A}\rightarrow\mathbb{R}$
The equation (\ref{equation}) gives the following condition for
the action on the connections $$g_i(\theta)\cdot
A_i=g_i(\theta)\circ A_i\circ g_i(\theta)^{-1}=A_i,$$ or
equivalently $$A_i\circ g_i(\theta)=g_i(\theta)\circ A_i,$$ that
is, the automorphism $g_i(\theta)$ is parallel with respect to the
connection $A_i$. Then we have $$d_{A_i}(\psi_i)=0.$$ That is, the
family $\{g(\theta)=(g_1(\theta),g_2(\theta))\}$ is generated by
an infinitesimal gauge transformation $\psi=(\psi_1,\psi_2)$ which
is covariantly constant, that is,
$$d_{A_1}(\psi_1)=d_{A_2}(\psi_2)=0$$ and with
$$[\psi,\varphi]=i\varphi.$$

%Since in our case the bundle $F=F_1\times F_2$ is a principal
%$\SO(1)\times\SO(n)$-Higgs bundle, the group $\mathcal{G}(F)$ of
%$\SO(1)\times\SO(n)$-invariant automorphisms of $F$ is isomorphic
%to $\{1\}\times\mathcal{G}(F_2)$ which is in correspondence with
%$\{1\}\times\Omega^0(F_2(\SO(n)))$. Then, the corresponding Lie
%algebra of infinitesimal gauge transformations is $\text{Lie
%}\mathcal{G}(F)=\Omega^0(F_2(\mathfrak{so}(n)))$.

\begin{proposition}\label{hodgebundles}
An $\SO_0(1,n)$-Higgs bundle
$(\mathcal{O},W,Q_W,\eta)\in\mathcal{M}$ with $\eta\neq 0$
represents a fixed point of the circle action if and only it is a
Hodge bundle (complex variation of Hodge structure), that is, if
and only if the vector bundles $W$ have a decomposition
$$W=\bigoplus_{r=-s}^s W_r,$$
with $W_r\cong (W^*)_{-r}$ and ${\psi_2}|_{W_r}=ir$ for an
infinitesimal gauge transformation $\psi_2$. The only piece of
Higgs field non-equal to zero is
$$\eta:W_{-1}\rightarrow\mathcal{O}\otimes K\text{ (and $\eta^\top:\mathcal{O}\rightarrow W_1\otimes K$).}$$
\end{proposition}

\begin{proof}
If $(\mathcal{O},W,Q_W,\eta)$ represents a smooth point of the
moduli space which is a critical point of $f$, then is is a fixed
point of the circle action. The condition $d_{A_2}(\psi_2)=0$ in
the context of Higgs bundles means that the infinitesimal gauge
transformation $\psi_2$ gives a decomposition
$$W=\bigoplus_r W_r,$$ where $\underline{r\in\mathbb{R}}$
and ${\psi_2}|_{W_r}=ir$. Moreover, since $\psi_2$ is locally in
$\mathfrak{so}(n)$, it satisfies $\psi_2=-\psi_2^\top$. If
$q_W:W\cong W^*$ is the isomorphism given by the orthogonal form
$Q_W$, we have $\psi_2^\top=q_W^{-1}\circ\psi_2^t\circ q_W$, and
for all $w\in W_r$ we have
$$\psi_2^t(q_W(w))=q_W(\psi_2^\top(w))=-q_W(\psi_2(w))=-irq_W(w),$$
that is,
$$w\in W_r\Leftrightarrow q_W(w)\in(W^*)_{-r}.$$ Hence, we have
an isomorphism $W_r\cong (W^*)_{-r}$.

If $w\in W_r$ and $w'\in W_l$,
$$Q_W(\psi_2(w),w')=Q_W(irw,w')=irQ_W(w,w')$$ and, on the other
hand,
$$Q_W(\psi_2(w),w')=Q_W(w,\psi_2^\top(w'))=Q_W(w,-\psi_2(w'))=Q_W(w,-ilw')=-ilQ_W(w,w'),$$
that is, $$i(r+l)Q_W(w,w')=0.$$ Then, all the $W_l$ are orthogonal
to $V_r$ (including $l=r$) under $Q_W$ except $l=-r$. Since $Q_W$
is non-degenerate,
$$Q_W(w,w')=0\text{ for all }w'\in W\Rightarrow w=0,$$ and then, given $0\neq w\in
W_r$, there is a $w'\in W$ with $Q_W(w,w')\neq 0$, that is, a
$w'\in W_{-r}$. Then
$$W=\bigoplus_{r=-s}^s W_r.$$

We also know that the endomorphism $\psi_2$ is trace free, then
$$0=\Tr(\psi_2)=i\sum_{r=-s}^s
r\rk(W_r)\Leftrightarrow\sum_{r=-s}^s r\rk(W_r)=0.$$

The condition $[\psi,\varphi]=i\varphi$ for the solution
$(A,\varphi)$ is equivalent in this context to
$$-\eta\psi_2=i\eta.$$ If $w\in W_r$, we have
$$-\eta(\psi_2(w))=-\eta(irw)=-ir\eta(w)=i\eta(w)\Leftrightarrow r={-1}\text{ }(\eta\neq 0),$$
and we conclude.
\end{proof}

From Theorem \ref{hodgebundles} together with Proposition
\ref{fixedpoints} we have that if $(\mathcal{O},W,Q_W,\eta)$ is an
$\SO_0(1,n)$-Higgs bundle which represents a smooth point of the
moduli space, it is a critical point of the Hitchin function if
and only if it is a Hodge bundle, but observe that not every Hodge
bundle represents a smooth point.

\section{Smooth minima}

In this section we study the smooth minima of the Hitchin function
in the moduli space of $\SO_0(1,n)$-Higgs bundles.

Let $(E,\varphi)$ be an $\SO_0(1,n)$-Higgs bundle and let
$(E_{\SO(n+1,\mathbb{C})},\varphi)$ be the associated
$\SO(n+1,\mathbb{C})$-Higgs bundle. Consider also the tuple
$(\mathcal{O},W,Q_W,\eta)$ corresponding to $(E,\varphi)$ and the
triple $(\textbf{E},Q,\varphi)$ corresponding to
$(E_{\SO(n+1,\mathbb{C})},\varphi)$. We have that
\begin{eqnarray*}
E_{\SO(n+1,\mathbb{C})}(\mathfrak{so}(n+1,\mathbb{C}))&=&\{f\in\End(\textbf{E})\mid
f+f^\top=0\}=\mathfrak{so}(\textbf{E}),\\
  E(\mathfrak{h}^\mathbb{C})&=&\{\left(%
\begin{array}{cc}
  0 & 0 \\
  0 & f_4 \\
\end{array}%
\right)\in\End(\textbf{E})\mid f_4+f_4^\top=0\}\\
   && \cong\mathfrak{so}(\mathcal{O})\oplus\mathfrak{so}(W)\subset\End(\mathcal{O})\oplus\End(W),\\
E(\mathfrak{m}^\mathbb{C})&=&\{\left(%
\begin{array}{cc}
  0 & f_2 \\
  -f_2^\top & 0 \\
\end{array}%
\right)\in\End(\textbf{E})\}\cong\Hom(W,\mathcal{O}).
\end{eqnarray*} In fact,
$$E_{\SO(n+1,\mathbb{C})}(\mathfrak{so}(n+1,\mathbb{C}))=E(\mathfrak{h}^\mathbb{C})\oplus
E(\mathfrak{m}^\mathbb{C}),$$ which is induced by the Cartan
decomposition of the Lie algebra $\mathfrak{so}(n+1,\mathbb{C})$.

If $(\mathcal{O},W,Q_W,\eta)$ is a Hodge bundle, from Proposition
\ref{hodgebundles} we have that there is an infinite\-si\-mal
gauge transformation $\psi_2$ such that
$$W=\bigoplus_{r=-s}^s W_r,$$
with $W_r\cong (W^*)_{-r}$, ${\psi_2}|_{W_r}=ir$ and
$$\eta:W_{-1}\rightarrow\mathcal{O}\otimes K.$$
This decompositions of $W$ gives decompositions
$$\begin{array}{rcl}
\End(W)&=&\displaystyle{\bigoplus_{k=-2s}^{2s}(\bigoplus_{i-j=k}\Hom(W_j,W_i)),}\\
\Hom(W,\mathcal{O})&=&\displaystyle{\bigoplus_{k=-s}^{s}\Hom(W_k,\mathcal{O}).}
\end{array}$$

If $g_{k,l}\in\Hom(W_k,W_l)$, using the isomorphism $q_W$ induced
by the orthogonal form $Q_W$ we have that the diagram
$$\xymatrix{W_l^*\ar[r]^{g_{k,l}^t}\ar[d]^\cong&W_k^*\ar[d]^\cong\\
W_{-l}\ar[r]^{g_{k,l}^\top}&W_{-k},}$$ is commutative, and then,
the skew-symmetry in $\mathfrak{so}(W)\subset\End(W)$ is
equivalent to the condition $g_{-l,-k}+g_{k,l}^\top=0$, that is,
the following sets are related by skew-symmetry
$$\begin{array}{rcl}
  g_{k,l} & \longleftrightarrow & -g_{k,l}^\top,\\
  \Hom(W_k,W_l) & \longleftrightarrow & \Hom(W_{-l},W_{-k}).
\end{array}$$ Observe that when $k=l$, the endomorphism
and $g_{k,l}$ is skew-symmetric. $\text{Analogously}$, in
$E(\mathfrak{m}^\mathbb{C})$ we have the relation:
$$\begin{array}{rcl}
  h_k & \longleftrightarrow & -h_k^\top,\\
  \Hom(W_k,\mathcal{O}) & \longleftrightarrow & \Hom(\mathcal{O},W_{-k}).
\end{array}$$

Then, the decomposition of $W$ also induce decompositions of
$E(\mathfrak{h}^\mathbb{C})\cong\mathfrak{so}(\mathcal{O})\oplus\mathfrak{so}(W)\cong\mathfrak{so}(W)$
and $E(\mathfrak{m}^\mathbb{C})\cong\Hom(W,\mathcal{O})$, which
gives a decomposition of the deformation complex of Section
\ref{deformationsection}:
$$C^\bullet(\mathcal{O},W,Q_W,\eta):\mathfrak{so}(W)\rightarrow \Hom(W,\mathcal{O})\otimes K,$$
given by
$$C^\bullet(\mathcal{O},W,Q_W,\eta)=
\bigoplus_k C_k^\bullet(\mathcal{O},W,Q_W,\eta),$$ where
$C_k^\bullet(\mathcal{O},W,Q_W,\eta)$ are the subcomplexes
$$C_k^\bullet(\mathcal{O},W,Q_W,\eta):\mathfrak{so}(W)_k\rightarrow \Hom(W,\mathcal{O})_{k+1}\otimes
K.$$ This induces a decomposition of the infinitesimal deformation
space given by
$$\mathbb{H}^1(C^\bullet(\mathcal{O},W,Q_W,\eta))=\bigoplus_k\mathbb{H}^1(C^\bullet_k(\mathcal{O},W,Q_W,\eta)).$$

A convenient reference for the following results is
Garc\'ia-Prada, Gothen and Mundet i Riera \cite{G-PGM}.

\begin{proposition}\label{localminima}
Let $(\mathcal{O},W,Q_W,\eta)$ be an $\SO_0(1,n)$-Higgs bundle
which represents a smooth point of the moduli space $\mathcal{M}$
and which is a critical point of $f$. The hyperco\-ho\-mo\-lo\-gy
group $\mathbb{H}^1(C^\bullet_k(\mathcal{O},W,Q_W,\eta))$ is
isomorphic to the eigenspace of the Hessian of $f$ with eigenvalue
$-k$. Then, $(\mathcal{O},W,Q_W,\eta)$ corresponds to a local
minimum of $f$ if and only if
$$\mathbb{H}^1(C_k^\bullet(\mathcal{O},W,Q_W,\eta))=0\text{ for }k>0.$$
\end{proposition}

%If $[E,\varphi]\in\mathcal{M}^s$ ($s$ for smooth) is a critical
%point of $f$, then $[E,\varphi]\in\mathcal{M}_l$ for some critical
%manifold $\mathcal{M}_l$ and we define the Morse index at the
%point $[E,\varphi]$ as the Morse index of $\mathcal{M}_l$. The
%first part of this proposition alow us to calculate the Morse
%index of $[E,\varphi]$, which is
%$$\indice=\displaystyle{\sum_{k>0}\dim_\mathbb{R}\mathbb{H}^1(C_k^\bullet(E,\varphi))}.$$
%If $[E,\varphi]\in\mathcal{M}^s$ is a minimum of $f$, that is
%$[E,\varphi]\in\mathcal{N}$, then
%$$\mathbb{H}^1(C_k^\bullet(E,\varphi))=0\text{ for }k>0,$$ and the Morse
%index at $[E,\varphi]$ is null.

To give a criterion for deciding when the hypercohomology
$\mathbb{H}(C^\bullet_k(E,\varphi))$ vanishes, we use the Euler
characteristic of the complex
$C^\bullet_k(\mathcal{O},W,Q_W,\eta)$. If we denoted by
$\text{h}^i(\mathcal{O},W,Q_W,\eta)$ the dimension of the
hypercohomology group
$\mathbb{H}^i(C_k^\bullet(\mathcal{O},W,Q_W,\eta))$, the Euler
cha\-rac\-te\-ris\-tic is defined by
$$\chi(C^\bullet_k(\mathcal{O},W,Q_W,\eta))=\text{h}^0(C^\bullet_k(\mathcal{O},W,Q_W,\eta))-
\text{h}^1(C^\bullet_k(\mathcal{O},W,Q_W,\eta))+\text{h}^2(C^\bullet_k(\mathcal{O},W,Q_W,\eta)).$$

\begin{proposition}
Let $(\mathcal{O},W,Q_W,\eta)$ be an $\SO_0(1,n)$-Higgs bundle
which represents a fixed point under the circle action on
$\mathcal{M}$. Then
$$\chi(C^\bullet_k(\mathcal{O},W,Q_W,\eta))\leq 0,$$ and equality holds if and
only if the map
$$C_k^\bullet(\mathcal{O},W,Q_W,\eta):\mathfrak{so}(W)_k\rightarrow \Hom(W,\mathcal{O})_{k+1}\otimes
K$$ is an isomorphism.
\end{proposition}

If $(\mathcal{O},W,Q_W,\eta)$ represents a smooth
$\SO_0(1,n)$-Higgs bundle, using Corollary \ref{H0yH2}, we have
that
$$\mathbb{H}^0(C^\bullet_k(\mathcal{O},W,Q_W,\eta))=\mathbb{H}^2(C^\bullet_k(\mathcal{O},W,Q_W,\eta))=0,$$ and
then,
$$-\chi(C^\bullet_k(\mathcal{O},W,Q_W,\eta))=\text{h}^1(C^\bullet_k(\mathcal{O},W,Q_W,\eta)),$$
for all $k$. Applying Proposition \ref{localminima}, we have the
following criterion for local minima of $f$.

\begin{proposition}\label{criteriumofminima}
Let $(\mathcal{O},W,Q_W,\eta)$ be an $\SO_0(1,n)$-Higgs bundle
which represents a smooth point of $\mathcal{M}$ and which is a
critical point of $f$. Then it represents a local minimum if and
only if
$$C_k^\bullet(\mathcal{O},W,Q_W,\eta):\mathfrak{so}(W)_k\rightarrow \Hom(W,\mathcal{O})_{k+1}\otimes
K$$ is an isomorphism for all $k>0$.
\end{proposition}

%This criterion can be used only when the $\SO_0(1,2m+1)$-Higgs
%bundle represents a smooth point of the moduli space. If
%$(\mathcal{O},W,Q_W,\eta)$ does not represent a smooth point, it
%is shown in \cite{Hi2} that if
%$\mathbb{H}^1(C^\bullet_k(\mathcal{O},W,Q_W,\eta))=0$ for all
%$k>0$ then $(\mathcal{O},W,Q_W,\eta)$ is a local minimum, but we
%will see in *******Proposition \ref{ruleout} that to have
%$$\mathbb{H}^1(C^\bullet_-(\mathcal{O},W,Q_W,\eta))=\bigoplus_{k>0}\mathbb{H}^1(C^\bullet_k(\mathcal{O},W,Q_W,\eta))\neq 0$$
%is not enough to prove that the point is not a minimum.

Applying this criterion we obtain the following result.

\begin{theorem}\label{SO(1,2m+1)smoothminima}
The smooth minima of the Hitchin function in the moduli space of
polystable $\SO_0(1,n)$-Higgs bundles with $n>2$ have zero Higgs
field.% There are no smooth minima of the Hitchin function with
%non-zero Higgs field in the moduli space of polystable
%$\SO_0(1,n)$-Higgs bundles with $n>2$.
\end{theorem}

\begin{proof}
Let $(\mathcal{O},W,Q_W,\eta)$ be a smooth point of the moduli
space with $\eta\neq 0$ which is a minimum of the Hitchin
function. Since it is stable, we know from Theorem
\ref{polystable} and Proposition \ref{stable} and
\ref{strictlypolystable} that it decomposes as a sum of stable
$G_i$-Higgs bundles where $G_i=\SO_0(1,n_i)$ and $\SO(n_i)$ with
$n_i\neq 2$. Since $(\mathcal{O},W,Q_W,\eta)$ is a critical point,
from Proposition \ref{hodgebundles}, the $\SO_0(1,n_i)$-Higgs
bundle in the decomposition is of the form
$$W^i_{-1}\rightarrow\mathcal{O}\rightarrow W^i_1,$$ where
$n_i=2\rk(W^i)$ (and $0<\deg(W^i_{-1})\leq 2g-2$). Observe that,
since $\mathcal{O}$ and the other $\SO(n_i)$-Higgs bundles in the
decomposition are self-dual, then they have weight $0$.

Since $(\mathcal{O},W,Q_W,\eta)$ is a minimum of the Hitchin
function, using Proposition \ref{criteriumofminima}, the
subcomplex
$$C_{2}^\bullet(\mathcal{O},W,Q_W,\eta):\Lambda^2 W^i_1\rightarrow 0$$ has to be an isomorphism.
Then $\rk(W^i_1)=\rk(W^i_{-1})=1$.

Since the Hitchin function is additive with respect to the direct
sum and $(\mathcal{O},W,Q_W,\eta)$ is a minimum, each $G_i$-Higgs
bundle in the decomposition has to be a minimum on the
corresponding moduli space $\mathcal{M}(G_i)$ and a minimum as
$\SO_0(1,n_i)$-Higgs bundles. Using the criterion of Proposition
\ref{criteriumofminima}, we have that the $\SO_0(1,2)$-Higgs
bundle $(\mathcal{O},W^i_1\oplus W^i_{-1},\eta)$ is a minimum. The
summands corresponding to $\SO(n_i)$-Higgs bundles are minima,
because they have Higgs field equal to zero. Consider now the sum
of this Higgs bundle together with an $\SO(n_i)$-Higgs bundle
$(E,Q)$ in the decomposition of $(\mathcal{O},W,Q_W,\eta)$. (Since
$n>2$, there is at least one summand of this type). The subcomplex
$$C_{1}^\bullet(\mathcal{O},W^i_1\oplus W^i_{-1}\oplus E,\eta):\Hom(W^i_{-1},E)\rightarrow
0$$ is not an isomorphism and then $(\mathcal{O},W^i_1\oplus
W^i_{-1}\oplus E,\eta)$ is not a minimum. We get a contradiction
and we conclude that the Higgs field $\eta$ has to be equal to
zero.
\end{proof}

\section{Minima in the whole moduli space}

In the previous section we characterize the minima of the Hitchin
functional in the smooth locus of the moduli space of
$\SO_0(1,n)$-Higgs bundle. In this section we extend the
characterization to the whole moduli space for $n$ odd. This
allows us to solve the problem of counting the connected
components of $\mathcal{M}(\SO_0(1,n))$ with $n$ odd.

\begin{theorem}\label{SO(1,2m+1)minima}
All the minima of the Hitchin function in the moduli space of
polystable $\SO_0(1,n)$-Higgs bundles, with $n$ odd, have the
Higgs field equal to zero. %There are no minima of the Hitchin function with non-zero
%Higgs field in the moduli space of polystable $\SO_0(1,n)$-Higgs
%bundles with $n$ odd.
\end{theorem}

\begin{proof}
From Theorem \ref{SO(1,2m+1)smoothminima} we have that the smooth
minima of the Hitchin function in the moduli space of polystable
$\SO_0(1,n)$-Higgs bundles have zero Higgs field. In particular
this is true for $n$ odd.

$1.$ If $(\mathcal{O},W,Q_W,\eta)$ is a stable but non-simple
$\SO_0(1,n)$-Higgs bundle ($n$ odd) with $\eta\neq 0$ which is a
fixed point of the circle action, using Theorem
\ref{stablenonsimple} and Proposition \ref{hodgebundles}, we
obtain that it decomposes as a sum of a smooth minimum in
$\mathcal{M}(\SO_0(1,n_i))$ of the form
$$W^i_{-1}\rightarrow\mathcal{O}\rightarrow W^i_1,$$ together
with a sum of $\SO(n_i)$-Higgs bundles with $n_i\neq 2$ where at
least one has rank $n_i$ even. The first summand is necessary to
guarantee the condition $\eta\neq 0$ and the condition for the
rank $n_i$ to be even determines the non-simplicity of
$(\mathcal{O},W,Q_W,\eta)$.

As in the proof of Theorem \ref{SO(1,2m+1)smoothminima}, since the
Hitchin function $f$ is additive with respect to the direct sum,
if $(\mathcal{O},W,Q_W,\eta)$ is a minimum, each Higgs bundle in
its decomposition has to be a minimum on the corresponding moduli
space $\mathcal{M}(G_i)$ and a minimum as $\SO_0(1,n_i)$-Higgs
bundle.

Since $n\geq 3$, there is at least one $\SO(n_i)$-Higgs bundle in
the decomposition. If we consider this summand $(E,Q)$ together
with the one of the form $W_{-1}\rightarrow\mathcal{O}\rightarrow
W_1$, we obtain a smooth $\SO_0(1,n_i+2)$-Higgs bundle. Using the
same argument as in Theorem \ref{SO(1,2m+1)smoothminima} we deduce
that it is not a minimum (observe that $E\cong E^*$ and then it
has weight zero). This implies that $(V,Q_V,W,Q_W,\eta)$ is not a
minimum and we conclude.

$2.$ If $(\mathcal{O},W,Q_W,\eta)$ is a strictly polystable
$\SO_0(1,n)$-Higgs bundle ($n$ odd) with $\eta\neq 0$ which is a
fixed point of the circle action, it decomposes as a sum of a
smooth minimum in $\mathcal{M}(\SO_0(1,2))$ of the form
$$W_{-1}\rightarrow\mathcal{O}\rightarrow W_1,$$ together
with a sum of $\SO(n_i)$-Higgs bundles with and at least one
summand of one of the following types: an $\SO(2)$-Higgs bundle or
a $\U(n_i)$-Higgs bundle. The existence of this summand in the
decomposition is necessary to guarantee the strict polystability
of $(\mathcal{O},W,Q_W,\eta)$.

Since $n$ is odd, $n-2$ is also odd, and since
$$\U(n_i)\hookrightarrow\SO(2n_i)\hookrightarrow\SO_0(1,n-2),$$ with $2n_i$ even,
there is at least one $\SO(n_i)$-Higgs bundle $(E,Q)$ in the
decomposition (and $n_i$ is odd).

As in the stable but non-simple case, if we consider this summand
$(E,Q)$ together with the one of the form
$W^i_{-1}\rightarrow\mathcal{O}\rightarrow W^i_1$, we obtain a
smooth $\SO(1,n_i+2)$-Higgs bundle which is not a minimum and we
conclude.
\end{proof}

\begin{remark}
If $n$ is even, we can not guarantee the existence of an
$\SO(n_i)$-Higgs bundle in the decomposition in the second part of
the proof and then this result can not be generalized to the even
case.
\end{remark}

Using the characterization of the minima given by Theorem
\ref{SO(1,2m+1)minima} we solve the problem of counting the
connected components of the moduli space $\mathcal{M}(\SO_0(1,n))$
with $n$ odd.

\begin{theorem}\label{SO(1,2m+1)components}
The moduli space of $\SO_0(1,n)$-Higgs bundles when $n>1$ is odd
has $2$ connected components.
\end{theorem}

\begin{proof}
The topological invariant associated to an $\SO_0(1,n)$-Higgs
bundle $(\mathcal{O},W,Q_W,\eta)$ with $n\geq 3$ is the
Stiefel-Whitney class
$w_2\in\pi_1(\SO(n,\mathbb{C}))\cong\mathbb{Z}_2=\{0,1\}$. From
Theorem \ref{SO(1,2m+1)minima} we have that, when $n$ is odd,
there are no minima of the Hitchin function with non-zero Higgs
field, and then $\mathcal{M}(\SO_0(1,n))$ ($n$ odd) is the
disjoint union of the moduli spaces $\mathcal{M}_0(\SO_0(1,n))$
and $\mathcal{M}_1(\SO_0(1,n))$, which are connected.
\end{proof}

\end{document}